\documentclass[a4paper, 11pt]{article}

\usepackage[utf8]{inputenc}
\usepackage[OT1]{fontenc}
\usepackage{lmodern}     
\usepackage[english]{babel}
\usepackage{cite, hyphenat, hyperref}

\usepackage{amscd, amsfonts, amsmath}
\usepackage{amssymb, amsthm}
\usepackage{bbm, mathrsfs}
\usepackage{enumerate}
\usepackage{calc}
\usepackage[shortlabels]{enumitem}
\usepackage{makeidx}

\usepackage{nccfoots}
\usepackage{appendix}
\usepackage[a4paper, top=4cm, bottom=4cm, left=3cm, right=3cm]{geometry}
\allowdisplaybreaks

\numberwithin{equation}{section}
\theoremstyle{plain}

\newtheorem{Lemma}{Lemma}[section]
\newtheorem{Proposition}[Lemma]{Proposition}
\newtheorem{Theorem}[Lemma]{Theorem}
\newtheorem{Corollary}[Lemma]{Corollary}

\theoremstyle{definition}

\newtheorem{Example}[Lemma]{Example}
\newtheorem{Examples}[Lemma]{Examples}
\newtheorem{Remark}[Lemma]{Remark}

\def\cadlag{c\`{a}dl\`{a}g }

\def\oe{\"{o}}

\begin{document}

\title{\textbf{Support characterization for regular path-dependent stochastic Volterra integral equations}}

\author{Alexander Kalinin\footnote{Department of Mathematics, Imperial College London, United Kingdom. {\tt alex.kalinin@mail.de}. The author gratefully acknowledges support from Imperial College through a Chapman fellowship.}}

\date{\today}
\maketitle

\begin{abstract}
We consider a stochastic Volterra integral equation with regular path-dependent coefficients and a Brownian motion as integrator in a multidimensional setting. Under an imposed absolute continuity condition, the unique solution is a semimartingale that admits almost surely H\oe lder continuous paths. Based on functional It{\^o} calculus, we prove that the support of its law in the H\oe lder norm can be described by a flow of mild solutions to ordinary integro-differential equations that are constructed by means of the vertical derivative of the diffusion coefficient.
\end{abstract}

{\bf Keywords:} support of a measure, path-dependent Volterra process, functional Volterra integral equation, functional It{\^o} calculus, vertical derivative, H\oe lder space.

\section{Support representations via flows}

The support of the law of a continuous stochastic process consists of all continuous paths around any neighborhood the process may remain with positive probability. Determining this class of paths for a diffusion process, viewed as solution to a stochastic differential equation (SDE), establishes a relation between the coefficients of the equation and the law of its solution.  

In the pioneering work of Stroock and Varadhan\cite{StroockVaradhanSupp}, the support of the law of a diffusion process is characterized by an associated flow of classical solutions to ordinary differential equations. While Aida~\cite{AidaSuppHilbert} generalizes the time-homogeneous case to a Hilbert space, allowing for an infinite dimension, Gy\oe ngy and Pr\oe hle~\cite{GyoengyProehleSupp} deal with coefficients that are of affine growth and not necessarily bounded. Moreover, Pakkanen~\cite{PakkanenSupp} provides sufficient conditions for a stochastic integral to have the full support property. 

An extension of the Stroock-Varadhan support theorem to any $\alpha$-H\oe lder norm, where $\alpha\in (0,1/2)$, is given in Ben Arous et al.~\cite{LedouxEtAlHoelder}. The case of time-homogeneous coefficients was independently proven by Millet and Sanz-Sol{\'e}\cite{MilletSoleHoelder} and later extended to a parabolic stochastic partial differential equation (SPDE) in Bally et al.~\cite{BallyEtAlSPDE}. By using the vertical derivative as functional space derivative and generalizing the approach in~\cite{MilletSoleHoelder} with the relevant Girsanov changes of measures, a path-dependent version of the Stroock-Varadhan support theorem in H\oe lder norms was recently derived in~\cite{ContKalininSupp}. The contribution of this article is to extend this support characterization to stochastic Volterra integral equations with regular path-dependent coefficients by providing a flow of mild solutions to ordinary integro-differential equations.  

Let $r,T\geq 0$ with $r < T$ and $d,m\in\mathbb{N}$. We work with the separable Banach space $C([0,T],\mathbb{R}^{m})$ of all $\mathbb{R}^{m}$-valued continuous paths on $[0,T]$, endowed with the supremum norm given by $\|x\|_{\infty}=\sup_{t\in [0,T]}|x(t)|$, where $|\cdot|$ is used as absolute value function, Euclidean norm or Hilbert-Schmidt norm. Throughout, $\hat{x}\in C([0,T],\mathbb{R}^{m})$ and
\[
b:[r,T]^{2}\times C([0,T],\mathbb{R}^{m})\rightarrow\mathbb{R}^{m}\quad\text{and}\quad \sigma:[r,T]^{2}\times C([0,T],\mathbb{R}^{m})\rightarrow\mathbb{R}^{m\times d}
\]
are two product measurable maps that are \emph{non-anticipative} in the sense that they satisfy $b(t,s,x)$ $= b(t,s,x^{s})$ and $\sigma(t,s,x) = \sigma(t,s,x^{s})$ for all $s,t\in [r,T]$ with $s\leq t$ and each $x\in C([0,T],\mathbb{R}^{m})$, where $x^{s}$ denotes the path $x$ stopped at time $s$. 

On a filtered probability space $(\Omega,\mathscr{F},(\mathscr{F}_{t})_{t\in [0,T]},P)$ that satisfies the usual conditions and which possesses a standard $d$-dimensional $(\mathscr{F}_{t})_{t\in [0,T]}$-Brownian motion $W$, we consider the following path-dependent stochastic Volterra integral equation:
\begin{equation}\label{SVIE}
X_{t} = X_{r} + \int_{r}^{t}b(t,s,X)\,ds + \int_{r}^{t}\sigma(t,s,X)\,dW_{s}\quad\text{a.s.}
\end{equation}
for $t\in [r,T]$ with initial condition $X_{q} = \hat{x}(q)$ for $q\in [0,r]$ a.s. An absolute continuity and affine growth condition on the coefficients $b$ and $\sigma$ ensure that any solution to~\eqref{SVIE} is a semimartingale with delayed H\oe lder continuous trajectories.

In fact, for each $\alpha\in (0,1]$ let $C_{r}^{\alpha}([0,T],\mathbb{R}^{m})$ represent the non-separable Banach space of all $x\in C([0,T],\mathbb{R}^{m})$ that are $\alpha$-H\oe lder continuous on $[r,T]$, endowed with the \emph{delayed $\alpha$-H\oe lder norm} given by
\begin{equation}\label{Hoelder Norm}
\|x\|_{\alpha,r}:=\|x^{r}\|_{\infty} + \sup_{s,t\in [r,T]:\,s\neq t} \frac{|x(s) - x(t)|}{|s-t|^{\alpha}}.
\end{equation}
By convenience, we set $C_{r}^{0}([0,T],\mathbb{R}^{m}):=C([0,T],\mathbb{R}^{m})$ and $\|\cdot\|_{0,r}:=\|\cdot\|_{\infty}$. Then, under the conditions stated below, there is a unique strong solution to~\eqref{SVIE} whose sample paths belong a.s.~to the \emph{delayed H\oe lder space} $C_{r}^{\alpha}([0,T],\mathbb{R}^{m})$ for any $\alpha\in (0,1/2)$.

For $p\geq 1$ consider the separable Banach space $W_{r}^{1,p}([0,T],\mathbb{R}^{m})$ of all $x\in C([0,T],\mathbb{R}^{m})$ that are absolutely continuous on $[r,T]$ with a $p$-fold Lebesgue-integrable weak derivative $\dot{x}$, equipped with the \emph{delayed Sobolev $L^{p}$-norm} defined by
\begin{equation}\label{Sobolev Norm}
\|x\|_{1,p,r}:=\|x^{r}\|_{\infty} + \bigg(\int_{r}^{t}|\dot{x}(s)|^{p}\,ds\bigg)^{1/p}.
\end{equation}
Then it holds that $W_{r}^{1,p}([0,T],\mathbb{R}^{m})\subsetneq C_{r}^{1/q}([0,T],\mathbb{R}^{m})$ and $\|x\|_{1/q,r}\leq \|x\|_{1,p,r}$ for all $x\in W_{r}^{1,p}([0,T],\mathbb{R}^{m})$ whenever $p > 1$ and $q$ is its dual exponent. By allowing infinite values, we extend the definitions of $\|\cdot\|_{\infty}$ and $\|\cdot\|_{\alpha,r}$ at~\eqref{Hoelder Norm} to each path $x:[0,T]\rightarrow\mathbb{R}^{m}$ and the definition of $\|\cdot\|_{1,p,r}$ at~\eqref{Sobolev Norm} to every $x\in W_{r}^{1,1}([0,T],\mathbb{R}^{m})$.

Based on the non-separable Banach space $D([0,T],\mathbb{R}^{m})$ of all $\mathbb{R}^{m}$-valued \cadlag paths on $[0,T]$, endowed with the supremum norm $\|\cdot\|_{\infty}$, we use the following pseudometric on $[r,T]\times D([0,T],\mathbb{R}^{m})$ given by
\[
d_{\infty}((t,x),(s,y)):= |t-s|^{1/2} + \|x^{t} - y^{s}\|_{\infty}.
\]
Then a functional on this Cartesian product that is $d_{\infty}$-continuous is also non-anticipative and Lipschitz continuity relative to $d_{\infty}$ merely requires $1/2$-H\oe lder continuity in the time variable.

Let us now state the conditions under which the support theorem holds. By refering to \emph{horizontal} and \emph{vertical differentiability} of non-anticipative functionals from~\cite{ContFournieFunc, Dupire}, we in particular require that certain time and path space components of $\sigma$ are of class $\mathbb{C}^{1,2}$, a property to be recalled in Section~\ref{Section 2.1}. In this context, let $\partial_{s}$ be the horizontal, $\partial_{x}$ the vertical and $\partial_{xx}$ the second-order vertical differential operator. 

To have a simple notation if these first- and second-order space derivatives appear, we set $\|y\|:=(\sum_{k=1}^{m}\sum_{l=1}^{d}|y_{k,l}|^{2})^{1/2}$ if $y\in (\mathbb{R}^{1\times m})^{m\times d}$ or $y\in (\mathbb{R}^{m\times m})^{m\times d}$. Further, let $\mathbbm{I}_{d}$ be the identity matrix in $\mathbb{R}^{d\times d}$ and $A'$ denote the transpose of a matrix $A\in\mathbb{R}^{d\times m}$.

\begin{enumerate}[label=(C.\arabic*), ref= C.\arabic*, leftmargin=\widthof{(C.3)} + \labelsep]
\item\label{C.1} The map $[r,t)\times C([0,T],\mathbb{R}^{m})\rightarrow\mathbb{R}^{m\times d}$, $(s,x)\mapsto\sigma(t,s,x)$ is of class $\mathbb{C}^{1,2}$ for each $t\in (r,T]$, the maps $b(\cdot,s,x)$ and $\sigma(\cdot,s,x)$ are absolutely continuous on $[s,T]$ and $\partial_{x}\sigma(\cdot,s,x)$ is absolutely continuous on $(s,T]$ for any $s\in [r,T)$ and each $x\in C([0,T],\mathbb{R}^{m})$.

\item\label{C.2} The maps $\sigma$, $\partial_{x}\sigma$ and its weak time derivatives $\partial_{t}\sigma$, $\partial_{t}\partial_{x}\sigma$ are bounded. Further, there are $c,\lambda,\eta\geq 0$ and $\kappa\in [0,1)$ such that
\begin{align*}
|b(s,s,x)| + |\partial_{t}b(t,s,x)| &\leq c\big(1 + \|x\|_{\infty}^{\kappa}\big)\quad\text{and}\\
|\partial_{s}\sigma(t,s,x)| + \|\partial_{xx}\sigma(t,s,x)\|&\leq c\big(1 + \|x\|_{\infty}^{\eta}\big)\end{align*}
 for all $s,t\in [r,T)$ with $s < t$ and each $x\in C([0,T],\mathbb{R}^{m})$.
 
\item\label{C.3} There is $\lambda\geq 0$ satisfying $|b(s,s,x) - b(s,s,y)| + |\partial_{t}b(t,s,x) - \partial_{t}b(t,s,y)|$ $\leq \lambda\|x-y\|_{\infty}$ and
\begin{align*}
|\sigma(u,t,x)-\sigma(u,s,y)| + |\partial_{u}\sigma(u,t,x) - \partial_{u}\sigma(u,s,y)|\\
+ \|\partial_{x}\sigma(u,t,x) - \partial_{x}\sigma(u,s,y)\| &\leq\lambda d_{\infty}((t,x),(s,y))
\end{align*}
for any $s,t,u\in [r,T)$ with $s < t < u$ and every $x,y\in C([0,T],\mathbb{R}^{m})$.
\end{enumerate}

Under the assumption that $\sigma(t,\,\cdot,\cdot)$ is of class $\mathbb{C}^{1,2}$ on $[r,t)\times C([0,T],\mathbb{R}^{m})$ for each $t\in (r,T]$, we may introduce the map $\rho:[r,T]^{2}\times C([0,T],\mathbb{R}^{m})\rightarrow\mathbb{R}^{m}$, which serves as \emph{correction term}, coordinatewise by
\begin{equation}\label{Correction Term}
\rho_{k}(t,s,x) = \sum_{l=1}^{d}\partial_{x}\sigma_{k,l}(t,s,x)\sigma(s,s,x)e_{l},
\end{equation}
if $s < t$ and, $\rho_{k}(t,s,x):=0$, otherwise. Here, $\{e_{1},\dots,e_{d}\}$ stands for the standard basis of $\mathbb{R}^{d}$ and $[r,t)\times C([0,T],\mathbb{R}^{m})\rightarrow\mathbb{R}^{1\times m}$, $(s,x)\mapsto \partial_{x}\sigma_{k,l}(t,s,x)$ is the vertical derivative of the $(k,l)$-entry of the map $[r,t)\times C([0,T],\mathbb{R}^{m})\rightarrow\mathbb{R}^{m\times d}$, $(s,x)\mapsto \sigma(t,s,x)$ for each $t\in (r,T]$, every $k\in\{1,\dots,m\}$ and any $l\in\{1,\dots,d\}$.

Finally, to describe the support of the unique strong solution to~\eqref{SVIE} by a flow, we study the following Volterra integral equation associated to any $h\in W_{r}^{1,p}([0,T],\mathbb{R}^{d})$ with $p\geq 1$. Namely,
\begin{equation}\label{Support VIE}
x_{h}(t) = x_{h}(r) + \int_{r}^{t}(b-(1/2)\rho)(t,s,x_{h})\,ds + \int_{r}^{t}\sigma(t,s,x_{h})\,dh(s)
\end{equation}
for $t\in [r,T]$. By adding $\hat{x}$ as initial condition, the solution $x_{h}$ lies in the \emph{delayed Sobolev space} $W_{r}^{1,p}([0,T],\mathbb{R}^{m})$, since it can also be viewed as a \emph{mild solution} to an associated ordinary integro-differential equation, as concisely justified in Section~\ref{Section 2.2}.

\begin{Lemma}\label{Support Auxiliary Lemma}
Let~\eqref{C.1}-\eqref{C.3} be valid.
\begin{enumerate}[(i)]
\item Pathwise uniqueness holds for~\eqref{SVIE} and there is a unique strong solution $X$ such that $X^{r}=\hat{x}^{r}$ a.s. Further, $X$ is a semimartingale and $E[\|X\|_{\alpha,r}^{p}] < \infty$ for any $\alpha\in [0,1/2)$ and all $p\geq 1$.

\item For any $p\geq 1$ and each $h\in W_{r}^{1,p}([0,T],\mathbb{R}^{d})$, there is a unique solution $x_{h}$ to~\eqref{Support VIE} satisfying $x_{h}^{r} = \hat{x}^{r}$ and we have $x_{h}\in W_{r}^{1,p}([0,T],\mathbb{R}^{m})$. Moreover, the flow map
\begin{equation}\label{Flow Map}
W_{r}^{1,p}([0,T],\mathbb{R}^{d})\rightarrow W_{r}^{1,p}([0,T],\mathbb{R}^{m}),\quad h\mapsto x_{h}
\end{equation}
is Lipschitz continuous on bounded sets.
\end{enumerate}
\end{Lemma}

Having clarified matters of uniqueness, existence and regularity, let us now consider the main result of this paper. Namely, a \emph{support characterization} of solutions to~\eqref{SVIE} in delayed H\oe lder norms.

\begin{Theorem}\label{Support Theorem}
Let~\eqref{C.1}-\eqref{C.3} hold, $\alpha\in [0,1/2)$ and $p\geq 2$. Then the support of the image measure of the unique strong solution $X$ to~\eqref{SVIE} in $C_{r}^{\alpha}([0,T],\mathbb{R}^{m})$ is the closure of the set of all solutions $x_{h}$ to~\eqref{Support VIE}, where $h\in W_{r}^{1,p}([0,T],\mathbb{R}^{d})$. That is,
\begin{equation}\label{Support Characterization}
\mathrm{supp}(P\circ X^{-1}) = \overline{\{x_{h}\,|h\in W_{r}^{1,p}([0,T],\mathbb{R}^{d})\}}\quad\text{in $C_{r}^{\alpha}([0,T],\mathbb{R}^{m})$.}
\end{equation}
\end{Theorem}

\begin{Example}
Suppose that there are four product measurable maps $K_{b},K_{\sigma}:[r,T]^{2}\rightarrow\mathbb{R}$, $\overline{b}:[r,T]\times C([0,T],\mathbb{R}^{m})\rightarrow\mathbb{R}^{m}$ and $\overline{\sigma}:[r,T]\times C([0,T],\mathbb{R}^{m})\rightarrow\mathbb{R}^{m\times d}$ such that
\begin{equation*}
b(t,s,x) = K_{b}(t,s)\overline{b}(s,x)\quad\text{and}\quad \sigma(t,s,x) = K_{\sigma}(t,s)\overline{\sigma}(s,x)
\end{equation*}
for all $s,t\in [r,T]$ and any $x\in C([0,T],\mathbb{R}^{m})$ and let the following three conditions hold:
\begin{enumerate}[(1)]
\item The functions $K_{b}(\cdot,s)$ and $K_{\sigma}(\cdot,s)$ are differentiable for each $s\in [r,T)$. Further, $K_{b}$, $K_{\sigma}$, $\partial_{t}K_{b}$ and $\partial_{t}K_{\sigma}$ are bounded.

\item The map $\overline{\sigma}$ is of class $\mathbb{C}^{1,2}$ on $[r,T)\times C([0,T],\mathbb{R}^{m})$ and together with its vertical derivative $\partial_{x}\overline{\sigma}$ it is bounded and $d_{\infty}$-Lipschitz continuous.

\item There are $c,\eta,\lambda\geq 0$ and $\kappa\in [0,1)$ such that
\begin{align*}
|\overline{b}(s,x)|\leq c(1 + \|x\|_{\infty}^{\kappa}),\quad |\overline{b}(s,x) - \overline{b}(s,y)|&\leq \lambda\|x-y\|_{\infty},\\
|K_{\sigma}(u,t) - K_{\sigma}(u,s)| + |\partial_{u}K_{\sigma}(u,t) - \partial_{u}K_{\sigma}(u,s)|&\leq\lambda |s-t|^{1/2}\quad\text{and}\\
|\partial_{s}\overline{\sigma}(s,x)| + |\partial_{xx}\overline{\sigma}(s,x)|&\leq c(1 + \|x\|_{\infty}^{\eta})
\end{align*}
for all $s,t,u\in [r,T)$ with $s < t < u$ and each $x,y\in C([0,T],\mathbb{R}^{m})$.
\end{enumerate}
Then Theorem~\ref{Support Theorem} applies and in the specific case that $K_{b} = K_{\sigma} = 1$ it reduces to the support theorem in~\cite{ContKalininSupp} with the same regularity conditions.
\end{Example}

The structure of this paper is determined by the proof of the support theorem and can be comprised as follows. Section~\ref{Section 2} provides supplementary material and a H\oe lder convergence result that yields Theorem~\ref{Support Theorem} as a corollary. In detail, Section~\ref{Section 2.1} gives a concise overview of horizontal and vertical differentiability of non-anticipative functionals. Section~\ref{Section 2.2} relates the Volterra integral equation~\eqref{Support VIE} to an ordinary integro-differential equation and shows that solutions to~\eqref{SVIE} are semimartingales by using a stochastic Fubini theorem. In Section~\ref{Section 2.3} we consider the approach to prove the support theorem by introducing a more general setting and stating Theorem~\ref{Main Convergence Theorem}, the before mentioned convergence result.

Section~\ref{Section 3} derives relevant estimates to infer convergence in H\oe lder norm in moment. To be precise, Section~\ref{Section 3.1} gives a sufficient condition for a sequence of processes to converge in this sense by exploiting an explicit Kolmogorov-Chentsov estimate. In Section~\ref{Section 3.2} we introduce the relevant notations in the context of sequence of partitions and recall a couple of auxiliary moment estimates from~\cite{ContKalininSupp,MaoSDEandApp}. The purpose of Section~\ref{Section 3.3} is to deduce moment estimates for deterministic and stochastic Volterra integrals, generalizing the bounds from~\cite{ContKalininSupp}[Lemmas 20, 21 and Proposition 22].

Section~\ref{Section 4} is devoted to a variety of specific moment estimates and decompositions, preparing the proof of Theorem~\ref{Main Convergence Theorem}. At first, Section~\ref{Section 4.1} derives bounds for solutions to stochastic Volterra integral equations and gives two main decompositions, Proposition~\ref{Main Decomposition Proposition} and~\eqref{Main Remainder Decomposition}. Section~\ref{Section 4.2} handles the first two remainders appearing in~\eqref{Main Remainder Decomposition}. While the second can be directly estimated, the first relies on the functional It{\^o} formula in~\cite{ContFournieFuncIto}. Section~\ref{Section 4.3} intends to bound the third remainder in second moment, requiring another extensive decomposition. In Section~\ref{Section 5} we prove the convergence result and the support representation, including assertions on uniqueness, existence and regularity. 

\section{Preparations and a convergence result in second moment}\label{Section 2}

\subsection{Differential calculus for non-anticipative functionals}\label{Section 2.1}

We recall and discuss horizontal and vertical differentiability, as introduced in~\cite{ContFournieFunc, Dupire}. To this end, let $t\in (r,T]$ and $G$ be a non-anticipative functional on $[r,t)\times D([0,T],\mathbb{R}^{m})$ that is considered at a point $(s,x)$ of its domain:
\begin{enumerate}[(i)]
\item $G$ is \emph{horizontally differentiable} at $(s,x)$ if the function $[0,T-s)\rightarrow\mathbb{R}$, $h\mapsto G(s + h,x^{s})$ is differentiable at $0$. If this is the case, then $\partial_{s}G(s,x)$ denotes its derivative there.

\item $G$ is \emph{vertically differentiable} at $(s,x)$ if the function $\mathbb{R}^{m}\rightarrow\mathbb{R}$, $h\mapsto G(s,x+h\mathbbm{1}_{[s,T]})$ is differentiable at $0$. In this case, its derivative there is denoted by $\partial_{x}G(s,x)$.

\item $G$ is \emph{partially vertically differentiable} at $(s,x)$ if for any $k\in\{1,\dots,m\}$ the function $\mathbb{R}\rightarrow\mathbb{R}$, $h\mapsto G(s,x+h\overline{e}_{k}\mathbbm{1}_{[s,T]})$ is differentiable at $0$, where $\{\overline{e}_{1},\dots,\overline{e}_{m}\}$ is the standard basis of $\mathbb{R}^{m}$. In this event, $\partial_{x_{k}}G(s,x)$ represents its derivative there.
\end{enumerate}

So, $G$ is horizontally, vertically or partially vertically differentiable if it satisfies the respective property at any point of its domain. We observe that vertical differentiability entails partial vertical differentiability and $\partial_{x}G = (\partial_{x_{1}}G,\dots,\partial_{x_{m}}G)$.

We say that $G$ is twice vertically differentiable if it is vertically differentiable and the same is true for $\partial_{x}G$. We then set $\partial_{xx}G:=\partial_{x}(\partial_{x}G)$ and $\partial_{x_{k}x_{l}}G:=\partial_{x_{k}}(\partial_{x_{l}}G)$ for any $k,l\in\{1,\dots,m\}$. If in addition $\partial_{xx}G$ is $d_{\infty}$-continuous, then
\[
\partial_{x_{k}x_{l}}G = \partial_{x_{l}x_{k}}G\quad\text{for each $k,l\in\{1,\dots,m\}$,}
\]
by Schwarz's Lemma, showing that $\partial_{xx}G$ is symmetric. Moreover, we call $G$ \emph{of class $\mathbb{C}^{1,2}$} if it is once horizontally and twice vertically differentiable such that $G$, $\partial_{s}G$, $\partial_{x}G$ and $\partial_{xx}G$ are bounded on bounded sets and $d_{\infty}$-continuous.

Clearly, horizontal differentiability applies to functionals on $[r,t)\times C([0,T],\mathbb{R}^{m})$ as well by considering continuous paths only. Vertical differentiability, however, requires the evaluation along \cadlag paths. So, a functional $F$ on $[r,t)\times C([0,T],\mathbb{R}^{m})$ is of \emph{class $\mathbb{C}^{1,2}$} if it possesses an non-anticipative extension $G:[r,t)\times D([0,T],\mathbb{R}^{m})\rightarrow\mathbb{R}$ that satisfies this property. Then the restricted derivatives
\[
\partial_{x}F:=\partial_{x}G\quad\text{and}\quad \partial_{xx}F:=\partial_{xx}G\quad\text{on $[r,t)\times C([0,T],\mathbb{R}^{m})$}
\]
are well-defined, by Theorems 5.4.1 and 5.4.2 in~\cite{ContFuncKol}. That is, they do not dependent on the choice of the extension $G$. By combining these considerations with an absolute continuity condition, which ensures that only semimartingales appear, we can use the functional It{\^o} formula from~\cite{ContFournieFuncIto} to prove Proposition~\ref{Remainder Proposition 1}, a key ingredient when deriving~\eqref{Support Characterization}.

\begin{Examples}
(i) We suppose that $\alpha\in (0,1]$, $k\in\mathbb{N}$ and $\varphi:[r,t)\times (\mathbb{R}^{m})^{k}\rightarrow\mathbb{R}^{d}$, $(s,x)\mapsto\varphi(s,\overline{x}_{1},\dots,\overline{x}_{m})$ is $\alpha$-H\oe lder continuous. Let $t_{0},\dots,t_{k}\in [r,t)$ satisfy $t_{0} < \cdots < t_{k}$, then the $\mathbb{R}^{d}$-valued non-anticipative map $G$ on $[r,t)\times D([0,T],\mathbb{R}^{m})$ given by
\[
G(s,x) := \varphi(s,x(t_{0}\wedge s),\dots,x(t_{k}\wedge s))
\]
is bounded on bounded sets and $\alpha$-H\oe lder continuous with respect to $d_{\infty}$. Furthermore, if $\varphi$ is of class $C^{1,2}$ in the usual sense, then $G$ is of class $\mathbb{C}^{1,2}$, because it satisfies $\partial_{s}G(s,x)$ $= (\partial_{+} \varphi/\partial s)(s,x(t_{0}\wedge s),\dots, x(t_{k}\wedge s))$ and
\[
\partial_{x}G(s,x) = \sum_{j=0,\, s\leq t_{j}}^{k} D_{\overline{x}_{j}}\varphi(s,x(t_{0}\wedge s),\dots, x(t_{k}\wedge s))
\]
for any $s\in [r,t)$ and every $x\in D([0,T],\mathbb{R}^{m})$, where $\partial_{+}\varphi/\partial s$ denotes the right-hand time derivative of $\varphi$ and $D_{\overline{x}_{j}}\varphi$ the partial derivative of $\varphi$ with respect to the $j$-th space variable $\overline{x}_{j}\in\mathbb{R}^{m}$ for each $j\in\{1,\dots,k\}$.\smallskip

\noindent
(ii) Let $\alpha\in (0,1]$, $K:[0,t)\rightarrow\mathbb{R}$ be continuously differentiable and $\varphi$ be an $\mathbb{R}^{m\times d}$-valued Borel measurable bounded map on $[0,t)\times D([0,T],\mathbb{R}^{m})$ that is $\alpha$-H\oe lder continuous in $x\in D([0,T],\mathbb{R}^{m})$, uniformly in $s\in [0,t)$. Then the \emph{non-anticipative kernel integral map} $G:[r,t)\times D([0,T],\mathbb{R}^{d})\rightarrow\mathbb{R}^{m\times d}$ defined by
\[
G(s,x):=\int_{0}^{s}K(s-u)\varphi(u,x^{u})\,du
\]
is bounded and $\alpha$-H\oe lder continuous relative to $d_{\infty}$. In addition, if $\varphi$ is $d_{\infty}$-continuous, then $G$ is of class $\mathbb{C}^{1,2}$, since $\partial_{s}G(s,x) = K(0)\varphi(s,x) + \int_{0}^{s}\dot{K}(s-u)\varphi(u,x)\,du$ for each $s\in [r,t)$ and any $x\in D([0,T],\mathbb{R}^{m})$ and $\partial_{x}G = 0$.
\end{Examples}

\subsection{Ordinary integro-differential equations and semimartingales}\label{Section 2.2}

By utilizing an absolute continuity condition, we directly connect the Volterra integral equation~\eqref{Support VIE} to an ordinary integro-differential equation and check that any solution to~\eqref{SVIE} solves a stochastic differential equation, ensuring that it is a semimartingale.

Let us first briefly analyze~\eqref{Support VIE} for $h\in W_{r}^{1,1}([0,T],\mathbb{R}^{d})$, under the hypothesis that $\sigma(t,\cdot,\cdot)$ is of class $\mathbb{C}^{1,2}$ on $[r,t)\times C([0,T],\mathbb{R}^{m})$ for each $t\in (r,T]$. A \emph{solution} to~\eqref{Support VIE} is a path $x\in C([0,T],\mathbb{R}^{m})$ such that
\begin{align*}
&\int_{r}^{t}|(b-(1/2)\rho)(t,s,x)| + |\sigma(t,s,x)||\dot{h}(s)|\,ds\quad\text{and}\\
& x(t) = x(r) + \int_{r}^{t}(b-(1/2)\rho)(t,s,x)\,ds + \int_{r}^{t}\sigma(t,s,x)\,dh(s)
\end{align*}
for any $t\in [r,T]$, since the variation of $h$ on $[r,s]$ is given by $\int_{r}^{s}|\dot{h}(u)|\,du$ for all $s\in [r,t]$. If we now assume that~\eqref{C.1}-\eqref{C.3} are valid, then the $d_{\infty}$-Lipschitz continuity of the map $[r,t)\times C([0,T],\mathbb{R}^{m})\rightarrow\mathbb{R}^{1\times m}$, $(s,x)\mapsto \partial_{x}\sigma_{k,l}(t,s,x)$ entails that it admits a unique continuous extension to $[r,t]\times C([0,T],\mathbb{R}^{m})$ for any $t\in (r,T]$, each $k\in\{1,\dots,m\}$ and every $l\in\{1,\dots,d\}$.

In this case, we may define $\overline{\rho}:[r,T]^{2}\times C([0,T],\mathbb{R}^{m})\rightarrow\mathbb{R}^{m}$ coordinatewise by letting $\overline{\rho}_{k}(t,s,x)$ agree with the right-hand side in~\eqref{Correction Term}, if $s\leq t$, and setting $\overline{\rho}(t,s,x):=0$, otherwise. Then Fubini's theorem entails for each $x\in C([0,T],\mathbb{R}^{m})$ that
\begin{align}\nonumber
&\int_{r}^{t}(b - (1/2)\rho)(t,s,x)\,ds + \int_{r}^{t}\sigma(t,s,x)\,dh(s)\\\label{Mild Solution Definition}
& = \int_{r}^{t}(b-(1/2)\overline{\rho} + \sigma\dot{h})(s,s,x) + \int_{r}^{s}\partial_{s}(b-(1/2)\overline{\rho} + \sigma\dot{h})(s,u,x)\,du\,ds
\end{align}
for every $t\in [r,T]$. Consequently, the path $x$ solves~\eqref{Support VIE} if and only if it is a \emph{mild solution} to the path-dependent ordinary integro-differential equation
\begin{equation*}
\dot{x}(t) = (b-(1/2)\overline{\rho} + \sigma\dot{h})(t,t,x) + \int_{r}^{t}\partial_{t}(b - (1/2)\overline{\rho} + \sigma\dot{h})(t,s,x)\,ds
\end{equation*}
for $t\in [r,T]$. Since all appearing maps are integrable, this means that the increment $x(t)-x(r)$ agrees with~\eqref{Mild Solution Definition} for any $t\in [r,T]$. Let us now turn to the stochastic Volterra integral equation~\eqref{SVIE}, without imposing any conditions for the moment.

Thus, we let $\mathscr{C}([0,T],\mathbb{R}^{m})$ denote the completely metrizable topological space of all $(\mathscr{F}_{t})_{t\in [0,T]}$-adapted continuous processes $X:[0,T]\times\Omega\rightarrow\mathbb{R}^{m}$ and recall that a \emph{solution} to~\eqref{SVIE} is a process $X\in\mathscr{C}([0,T],\mathbb{R}^{m})$ such that
\begin{align*}
&\int_{r}^{t}|b(t,s,X)| + |\sigma(t,s,X)|^{2}\,ds < \infty\quad\text{a.s.~and}\\
&X_{t} = X_{r} + \int_{r}^{t}b(t,s,X)\,ds + \int_{r}^{t}\sigma(t,s,X)\,dW_{s}\quad\text{a.s.~for all $t\in [r,T]$.}
\end{align*}
For a process $\xi\in\mathscr{C}([0,T],\mathbb{R}^{m})$ that is independent of $W$ we let $(\mathscr{E}_{t}^{0})_{t\in [0,T]}$ be the natural filtration of the adapted continuous process $[0,T]\times\Omega\rightarrow\mathbb{R}^{2m}$, $(t,\omega)\mapsto (\xi_{t}^{r},W_{r\vee t} - W_{r})(\omega)$. That is, $\mathscr{E}_{t}^{0}=\sigma(\xi_{q}:q\in [0,t])$ for $t\in [0,r]$ and
\[
\mathscr{E}_{t}^{0}:=\mathscr{E}_{r}^{0}\vee\sigma(W_{s}-W_{r}:s\in [r,t])\quad\text{for $t\in (r,T]$.}
\]
In particular, $\mathscr{E}_{t}^{0} = \sigma(\xi_{0})\vee\sigma(W_{s}:s\in [0,t])$ for all $t\in [0,T]$ if there is no delay. Let $(\mathscr{E}_{t})_{t\in [0,T]}$ be the right-continuous filtration of the augmented filtration of $(\mathscr{E}_{t}^{0})_{t\in [0,T]}$. Then a solution $X$ to~\eqref{SVIE} satisfying $X^{r} = \xi^{r}$ a.s.~is called \emph{strong} if it is adapted to this complete filtration.

Finally, suppose that~\eqref{C.1} and~\eqref{C.2} hold. Then it follows from Fubini's theorem for stochastic integrals, stated in~\cite{VeraarFubini}[Theorem 2.2] for instance, that any $X\in\mathscr{C}([0,T],\mathbb{R}^{m})$ satisfies
\begin{equation*}
\int_{r}^{t}b(t,s,X)\,ds + \int_{r}^{t}\sigma(t,s,X)\,dW_{s} = \int_{r}^{t}B_{s}(X)\,ds + \int_{r}^{t}\sigma(s,s,X)\,dW_{s}
\end{equation*}
a.s.~for any $t\in [r,T]$, where the map $B:[r,T]\times\Omega\times\mathscr{C}([0,T],\mathbb{R}^{m})\rightarrow\mathbb{R}^{m}$, which is product measurable and depends on whole processes rather than trajectories, is given by
\[
B_{s}(Y)= b(s,s,Y) + \int_{r}^{s}\partial_{s}b(s,u,Y)\,du + \int_{r}^{s}\partial_{s}\sigma(s,u,Y)\,dW_{u}
\]
for every $s\in [r,T]$ a.s. This shows that $X$ solves~\eqref{SVIE} if and only if it is a solution to the path-dependent stochastic differential equation
\[
X_{t} = B_{t}(X)\,dt + \sigma(t,t,X)\,dW_{t}\quad\text{for $t\in [r,T]$.}
\]
Moreover, it is automatically a semimartingale in this case.

\subsection{Approach to the main result in a general setting}\label{Section 2.3}

After these preliminary considerations, we proceed as follows to establish the support theorem. For any $n\in\mathbb{N}$ let $\mathbb{T}_{n}$ be a partition of $[r,T]$ of the form $\mathbb{T}_{n}=\{t_{0,n},\dots,t_{k_{n},n}\}$ with $k_{n}\in\mathbb{N}$ and $t_{0,n},\dots,t_{k_{n},n}\in [r,T]$ such that $r= t_{0,n} < \cdots < t_{k_{n},n} = T$ and whose mesh $\max_{i\in\{0,\dots,k_{n}-1\}}(t_{i+1,n} - t_{i,n})$ is denoted by $|\mathbb{T}_{n}|$. We assume that the sequence $(\mathbb{T}_{n})_{n\in\mathbb{N}}$ of partitions is \emph{balanced} as defined in~\cite{ContDasQuad}, which means that there is $c_{\mathbb{T}}\geq 1$ such that
\begin{equation}\label{Partition Condition}
|\mathbb{T}_{n}| \leq c_{\mathbb{T}} \min_{i\in\{0,\dots,k_{n}-1\}} (t_{i,n} - t_{i-1,n})\quad\text{for all $n\in\mathbb{N}$.}
\end{equation}
For the estimation of one term in Proposition~\ref{Remainder Proposition 1}, when the functional It{\^o} formula is applied, we also require the following additional condition:
\begin{enumerate}[label=(C.\arabic*), ref= C.\arabic*, leftmargin=\widthof{(C.4)} + \labelsep]
\setcounter{enumi}{3}
\item\label{C.4} There is $\overline{c}_{\mathbb{T}} > 0$ such that $k_{n} |\mathbb{T}_{n}| \leq \overline{c}_{\mathbb{T}}$ for each $n\in\mathbb{N}$.
\end{enumerate}
However, unless explicitly stated, we shall not impose this condition. Moreover, we readily notice that any equidistant sequence of partitions satisfies both conditions.

Next, for any $k,n\in\mathbb{N}$ we are interested in the delayed linear interpolation of a map $x:[0,T]\rightarrow\mathbb{R}^{k}$ along $\mathbb{T}_{n}$. Namely, we define $L_{n}(x):[0,T]\rightarrow\mathbb{R}^{k}$ by $L_{n}(x)(t):=x(r\wedge t)$, if $t\leq t_{1,n}$, and
\begin{equation}\label{Linear Operator L}
L_{n}(x)(t) := x(t_{i-1,n}) + \frac{t-t_{i,n}}{t_{i+1,n} - t_{i,n}}(x(t_{i,n}) - x(t_{i-1,n})),
\end{equation}
if $t\in (t_{i,n},t_{i+1,n}]$ for some $i\in\{1,\dots,k_{n}-1\}$. Since $L_{n}(x)$ is piecewise continuously differentiable, it belongs to $W_{r}^{1,p}([0,T],\mathbb{R}^{k})$ for every $p\geq 1$, and by construction, the process $\null_{n}W:[0,T]\times\Omega\rightarrow\mathbb{R}^{d}$ defined via $\null_{n}W_{t} :=L_{n}(W)(t)$ is adapted.

Let us now assume that~\eqref{C.1}-\eqref{C.3} and Lemma~\ref{Support Auxiliary Lemma} hold. Then the support of $P\circ X^{-1}$ is included in the closure of $\{x_{h}\,|\,h\in W_{r}^{1,p}([0,T],\mathbb{R}^{d})\}$ in $C_{r}^{\alpha}([0,T],\mathbb{R}^{m})$ for $\alpha\in [0,1/2)$ and $p\geq 2$ if we can prove that
\begin{equation}\label{Hoelder Limit 1}
\lim_{n\uparrow\infty} P(\|x_{\null_{n}W} - X\|_{\alpha,r}\geq\varepsilon) = 0\quad\text{for any $\varepsilon > 0$.}
\end{equation}
Moreover, if for each $h\in W_{r}^{1,p}([0,T],\mathbb{R}^{d})$ there exists a sequence $(P_{h,n})_{n\in\mathbb{N}}$ of probability measures on $(\Omega,\mathscr{F})$ that are absolutely continuous to $P$ such that
\begin{equation}\label{Hoelder Limit 2}
\lim_{n\uparrow\infty} P_{h,n}(\|X - x_{h}\|_{\alpha,r}\geq\varepsilon) = 0\quad\text{for every $\varepsilon > 0$,}
\end{equation}
then the converse inclusion holds. The sufficiency of~\eqref{Hoelder Limit 1} and~\eqref{Hoelder Limit 2} follows from a basic result on the support of probabiilty measures, see \cite{ContKalininSupp}[Lemma 36] for example. To verify the validity of both limits, we consider a more general setting.

Let $\underline{B}$ be an $\mathbb{R}^{m}$-valued and $B_{H},\overline{B}$ and $\Sigma$ be $\mathbb{R}^{m\times d}$-valued non-anticipative product measurable maps on $[r,T]^{2}\times C([0,T],\mathbb{R}^{m})$. For any $n\in\mathbb{N}$ we study the path-dependent stochastic Volterra integral equation:
\begin{equation}
\begin{split}\label{Sequence VIE}
\null_{n}Y_{t} &= \null_{n}Y_{r} + \int_{r}^{t}\underline{B}(t,s,\null_{n}Y) + B_{H}(t,s,\null_{n}Y)\dot{h}(s) + \overline{B}(t,s,\null_{n}Y)\null_{n}\dot{W}_{s}\,ds\\
&\quad + \int_{r}^{t}\Sigma(t,s,\null_{n}Y)\,dW_{s}\quad\text{a.s.~for $t\in [r,T]$.}
\end{split}
\end{equation}
Provided that the map $[r,t)\times C([0,T],\mathbb{R}^{m})\rightarrow\mathbb{R}^{m\times d}$, $(s,x)\mapsto\overline{B}(t,s,x)$ is of class $\mathbb{C}^{1,2}$ for all $t\in (r,T]$, we introduce another path-dependent stochastic Volterra integral equation:
\begin{equation}
\begin{split}\label{General VIE}
Y_{t} &= Y_{r} + \int_{r}^{t}(\underline{B}+R)(t,s,Y) + B_{H}(t,s,Y)\dot{h}(s)\,ds\\
&\quad + \int_{r}^{t}(\overline{B} + \Sigma)(t,s,Y)\,dW_{s}\quad\text{a.s.~for $t\in [r,T]$}
\end{split}
\end{equation}
with the $\mathbb{R}^{m}$-valued non-anticipative product measurable map $R$ on $[r,T]^{2}\times C([0,T],\mathbb{R}^{m})$ given coordinatewise by
\begin{equation}\label{General Remainder Definition}
R_{k}(t,s,x)=\sum_{l=1}^{d}\partial_{x}\overline{B}_{k,l}(t,s,x)\big((1/2)\overline{B} + \Sigma)(s,s,x)e_{l},
\end{equation}
if $s < t$, and $R_{k}(t,s,x):=0$, otherwise. In particular,~\eqref{Sequence VIE} reduces to~\eqref{General VIE} in the case that $\overline{B} = 0$. We seek to show that if $\null_{n}Y$ and $Y$ are two continuous solutions to~\eqref{Sequence VIE} and~\eqref{General VIE}, respectively, satisfying $\null_{n}Y^{r} = Y^{r} = \hat{x}^{r}$ a.s.~for all $n\in\mathbb{N}$, then
\begin{equation}\label{General Hoelder Limit}
\lim_{n\uparrow\infty} E[\|\null_{n}Y - Y\|_{\alpha,r}^{2}] = 0.
\end{equation}
Thus, by choosing $\underline{B} = b - (1/2)\rho$, $B_{H} = 0$, $\overline{B} = \sigma$ and $\Sigma = 0$, we obtain~\eqref{Hoelder Limit 1}. If instead $\underline{B} = b$, $B_{H}=\sigma$, $\overline{B} = - \sigma$ and $\Sigma=\sigma$, then~\eqref{Hoelder Limit 2} is implied, as we will see. To derive the general convergence result~\eqref{General Hoelder Limit}, we introduce the following regularity conditions:
\begin{enumerate}[label=(C.\arabic*), ref= C.\arabic*, leftmargin=\widthof{(C.9)} + \labelsep]
\setcounter{enumi}{4}
\item\label{C.5} The map $[r,t)\times C([0,T],\mathbb{R}^{m})\rightarrow\mathbb{R}^{m\times d}$, $(s,x)\mapsto\overline{B}(t,s,x)$ is of class $\mathbb{C}^{1,2}$ for all $t\in (r,T]$, for any $F\in\{\underline{B},B_{H},\overline{B},\Sigma\}$ the map $F(\cdot,s,x)$ is absolutely continuous on $[s,T]$ and $\partial_{x}\overline{B}$ is absolutely continuous on $(s,T]$ for each $s\in [r,T)$ and any $x\in C([0,T],\mathbb{R}^{m})$.

\item\label{C.6} There are $c\geq 0$ and $\kappa\in [0,1)$ such that any two maps $F\in\{\underline{B},B_{H}\}$ and $G\in\{\overline{B},\Sigma\}$ satisfy $|F(s,s,x)| + |\partial_{t}F(t,s,x)|\leq c(1 + \|x\|_{\infty}^{\kappa})$ and
\begin{equation*}
|G(s,s,x)| + |\partial_{t}G(t,s,x)|\leq c
\end{equation*}
for each $s,t\in [r,T)$ with $s < t$ and every $x\in C([0,T],\mathbb{R}^{m})$.

\item\label{C.7} There exists $\lambda\geq 0$ such that $|\underline{B}(s,s,x) - \underline{B}(s,s,y)| + |\partial_{t}\underline{B}(t,s,x) - \partial_{t}\underline{B}(t,s,y)|$ $\leq \lambda\|x-y\|_{\infty}$ and for any $F\in\{B_{H},\overline{B},\Sigma\}$ it holds that
\begin{align*}
|F(u,t,x) - F(u,s,y)| + |\partial_{u}F(u,t,x) - \partial_{u}F(u,s,y)|&\leq \lambda d_{\infty}((t,x),(s,y))
\end{align*}
for each $s,t,u\in [r,T)$ with $s < t < u$ and every $x,y\in C([0,T],\mathbb{R}^{m})$.

\item\label{C.8} There are $\overline{c},\eta,\overline{\lambda}\geq 0$ such that $\|\partial_{x}\overline{B}(s,s,x)\| + \|\partial_{t}\partial_{x}\overline{B}(t,s,x)\|\leq\overline{c}$, $|\partial_{s}\overline{B}(t,s,x)|$ $ +\, \|\partial_{xx}\overline{B}(t,s,x)\|\leq\overline{c}\big(1 + \|x\|_{\infty}^{\eta}\big)$ and
\begin{align*}
\|\partial_{x}\overline{B}(u,t,x) - \partial_{x}\overline{B}(u,s,y)\|&\leq \overline{\lambda}d_{\infty}((t,x),(s,y))
\end{align*}
for any $s,t,u\in [r,T)$ with $s < t < u$ and each $x\in C([0,T],\mathbb{R}^{m})$.

\item\label{C.9} There exist $\overline{b}_{0}\in\mathbb{R}$ and a measurable function $\overline{b}:[r,T]\rightarrow\mathbb{R}$ such that $\int_{r}^{T}\overline{b}_{1}(s)^{2}\,ds $ $< \infty$ and $\overline{b}_{0}\overline{B}(t,s,x)$ $= \overline{b}(s)\Sigma(t,s,x)$ for every $s,t\in [r,T)$ with $s < t$ and each $x\in C([0,T],\mathbb{R}^{m})$.
\end{enumerate}

First, we question uniqueness, existence and regularity of solutions to~\eqref{Sequence VIE} and~\eqref{General VIE}. In this regard, let $\xi\in\mathscr{C}([0,T],\mathbb{R}^{m})$ and $(\null_{n}\xi)_{n\in\mathbb{N}}$ be a sequence in $\mathscr{C}([0,T],\mathbb{R}^{m})$.

\begin{Lemma}\label{Auxiliary Convergence Lemma}
Assume that~\eqref{C.5}-\eqref{C.7} are satisfied, $h\in W_{r}^{1,2}([0,T],\mathbb{R}^{d})$ and for each $n\in\mathbb{N}$ there is is $p > 2$ such that $E[\|\xi^{r}\|_{\infty}^{p} + \|\null_{n}\xi^{r}\|_{\infty}^{p}] < \infty$.
\begin{enumerate}[(i)]
\item Under~\eqref{C.9}, pathwise uniqueness holds for~\eqref{Sequence VIE} and there exists a unique strong solution $\null_{n}Y$ with $\null_{n}Y^{r} = \null_{n}\xi^{r}$ a.s.~for any $n\in\mathbb{N}$. Further, for each $p > 2$ and every $\alpha\in [0,1/2-1/p)$, there is $c_{\alpha,p} > 0$ such that
\[
E[\|\null_{n}Y\|_{\alpha,r}^{p}] \leq c_{\alpha,p}(1 + E[\|\null_{n}\xi^{r}\|_{\infty}^{p}])\quad\text{for all $n\in\mathbb{N}$.}
\]

\item If~\eqref{C.8} holds, then we have pathwise uniqueness for~\eqref{General VIE} and a unique strong solution $Y$ with $Y^{r} = \xi^{r}$ a.s.~ In this case, for each $p > 2$ and all $\alpha\in [0,1/2-1/p)$ there is $\overline{c}_{\alpha,p}> 0$ with $E[\|Y\|_{\alpha,r}^{p}]$ $\leq \overline{c}_{\alpha,p}(1 + E[\|\xi^{r}\|_{\infty}^{p}])$.
\end{enumerate}
\end{Lemma}

Finally, we consider a convergence result in H\oe lder norm in second moment.

\begin{Theorem}\label{Main Convergence Theorem}
Let~\eqref{C.4}-\eqref{C.9} hold, $h\in W_{r}^{1,2}([0,T],\mathbb{R}^{d})$ and $\alpha\in [0,1/2)$. Suppose that $\lim_{n\uparrow\infty} E[\|\null_{n}\xi^{r} - \xi^{r}\|_{\infty}^{2}]/|\mathbb{T}_{n}|^{2\alpha} = 0$ and there is $p > 2$ such that
\[
\alpha < 1/2 - 1/p\quad\text{and}\quad E[\|\xi^{r}\|_{\infty}^{p}] + \sup_{n\in\mathbb{N}} E[\|\null_{n}\xi^{r}\|_{\infty}^{(2\vee\eta)p}] < \infty.
\]
Let $\null_{n}Y$ and $Y$ be the unique strong solutions to~\eqref{Sequence VIE} and~\eqref{General VIE}, respectively, such that $\null_{n}Y^{r} = \null_{n}\xi^{r}$ and $Y^{r} = \xi^{r}$ a.s.~for all $n\in\mathbb{N}$, then
\begin{equation}\label{General Partition Limit}
\lim_{n\uparrow\infty} E[\max_{j\in\{0,\dots,k_{n}\}}|\null_{n}Y_{t_{j,n}} - Y_{t_{j,n}}|^{2}\big]/|\mathbb{T}_{n}|^{2\alpha} = 0.
\end{equation}
In particular,~\eqref{General Hoelder Limit} is satisfied. That is, $(\null_{n}Y)_{n\in\mathbb{N}}$ converges in the norm $\|\cdot\|_{\alpha,r}$ in second moment to $Y$.
\end{Theorem}

\section{Estimates for convergence in H\oe lder norm in moment}\label{Section 3}

\subsection{Convergence in moment along a sequence of partitions}\label{Section 3.1}

We consider a sufficient condition for a sequence of processes to convergence in the norm $\|\cdot\|_{\alpha,r}$ in $p$-th moment, where $\alpha\in [0,1]$ and $p\geq 1$. Its derivation relies on an explicit Kolmogorov-Chentsov estimate~\cite{ContKalininSupp}[Proposition~12]. 

Namely, let $X$ be an $\mathbb{R}^{m}$-valued right-continuous processes for which there are $c_{0}\geq 0$, $p\geq 1$ and $q > 0$ such that $E[|X_{s} - X_{t}|^{p}]\leq c_{0}|s-t|^{1+q}$ for all $s,t\in [r,T]$. Then it follows that
\begin{equation}\label{Kolmogorov-Chentsov Estimate}
E\bigg[\sup_{s,t\in [r,T]:\,s\neq t}\frac{|X_{s} - X_{t}|^{p}}{|s-t|^{\alpha p}}\bigg] \leq k_{\alpha,p,q}c_{0}(T-r)^{1 + q - \alpha p}
\end{equation}
for any $\alpha\in [0,q/p)$ with $k_{\alpha,p,q}:=2^{p+q}(2^{q/p-\alpha} -1)^{-p}$. In particular, if $q\leq p$, then $X$ itself, and not necessarily a modification, admits a.s.~$\alpha$-H\oe lder continuous paths on $[r,T]$.

\begin{Lemma}\label{Convergence along Partitions Lemma}
Let $(\null_{n}X)_{n\in\mathbb{N}}$ be a sequence of $\mathbb{R}^{m}$-valued right-continuous processes for which there are $c_{0}\geq 0$, $p\geq 1$ and $q > 0$ with $q\leq p$ such that
\[
E[|\null_{n}X_{s} - \null_{n}X_{t}|^{p}] \leq c_{0}|s-t|^{1+q}
\]
for all $n\in\mathbb{N}$, each $j\in\{0,\dots,k_{n}-1\}$ and any $s,t\in [t_{j,n},t_{j+1,n}]$. If $(\|\null_{n}X^{r}\|_{\infty})_{n\in\mathbb{N}}$ and $(\max_{j\in\{1,\dots,k_{n}\}} |\null_{n}X_{t_{j,n}}|/|\mathbb{T}_{n}|^{\alpha})_{n\in\mathbb{N}}$ converge in $p$-th moment to zero, then so does the sequence $(\|\null_{n}X\|_{\alpha,r})_{n\in\mathbb{N}}$ for every $\alpha\in [0,q/p)$.
\end{Lemma}

\begin{proof}
For given $n\in\mathbb{N}$ a case distinction yields that
\begin{align*}
\sup_{s,t\in [r,T]:\,s\neq t} \frac{|\null_{n}X_{s} - \null_{n}X_{t}|}{|s-t|^{\alpha}}&\leq 2\max_{j\in\{0,\dots,k_{n}-1\}}\sup_{s,t\in [t_{j,n},t_{j+1,n}]:\,s\neq t} \frac{|\null_{n}X_{s} - \null_{n}X_{t}|}{|s-t|^{\alpha}}\\
&\quad + \max_{i,j\in\{1,\dots,k_{n}\}:\,i\neq j} \frac{|\null_{n}X_{t_{i,n}} - \null_{n}X_{t_{j,n}}|}{|t_{i,n} - t_{j,n}|^{\alpha}}.
\end{align*}
By virtue of the Kolmogorov-Chentsov estimate~\eqref{Kolmogorov-Chentsov Estimate}, it holds that
\begin{equation*}
E\bigg[\max_{j\in\{0,\dots,k_{n}-1\}}\sup_{s,t\in [t_{j,n},t_{j+1,n}]:\,s\neq t}\frac{|\null_{n}X_{s} - \null_{n}X_{t}|^{p}}{|s-t|^{\alpha p}}\bigg] \leq k_{\alpha,p,q} c_{0}(T-r)|\mathbb{T}_{n}|^{q - \alpha p},
\end{equation*}
since $q > \alpha p$ and $\sum_{j=0}^{k_{n}-1}(t_{j+1,n} - t_{j,n}) = T-r$. Moreover, from condition~\eqref{Partition Condition} we infer that $|t_{i,n}-t_{j,n}|\geq |\mathbb{T}_{n}|/c_{\mathbb{T}}$ for all $i,j\in\{0,\dots,k_{n}\}$ with $i\neq j$. Hence,
\begin{equation*}
E\bigg[\max_{i,j\in\{1,\dots,k_{n}\}:\,i\neq j}\frac{|\null_{n}X_{t_{i,n}} - \null_{n}X_{t_{j,n}}|^{p}}{|t_{i,n} - t_{j,n}|^{\alpha p}}\bigg] \leq 2^{p-1} c_{\mathbb{T}}^{\alpha p} E\big[\max_{j\in\{1,\dots,k_{n}\}} |\null_{n}X_{t_{j,n}}|^{p}\big]/|\mathbb{T}_{n}|^{\alpha p}
\end{equation*}
and the claim follows from the definition of the norm $\|\cdot\|_{\alpha,r}$.
\end{proof}

\subsection{Sequential notation and auxiliary moment estimates}\label{Section 3.2}

Let us introduce relevant notations related to the sequence of partitions $(\mathbb{T}_{n})_{n\in\mathbb{N}}$. For fixed $n\in\mathbb{N}$ and $t\in [r,T)$, we choose $i\in\{0,\dots,k_{n}-1\}$ such that $t\in [t_{i,n},t_{i+1,n})$ and set
\[
\underline{t}_{n}:=t_{(i-1)\vee 0,n},\quad t_{n}:=t_{i,n}\quad\text{and}\quad \overline{t}_{n}:=t_{i+1,n}.
\]
Verbalized, $\underline{t}_{n}$ is the predecessor of $t_{n}$ relative to $\mathbb{T}_{n}$, provided $i\neq 0$, and $\overline{t}_{n}$ is the successor of $t_{n}$. For the sake of completeness, let $\underline{T}_{n}:=t_{k_{n-1},n}$, $T_{n}:=T$ and $\overline{T}_{n}:=T$. Further, for $i\in\{0,\dots,k_{n}\}$ we set
\[
\Delta t_{i,n}:=t_{i,n} - t_{(i-1)\vee 0,n}\quad\text{and}\quad \Delta W_{t_{i,n}}:=W_{t_{i,n}} - W_{t_{(i-1)\vee 0,n}}.
\]

For $p\geq 1$ we recall an interpolation error estimate in supremum for stochastic processes in $p$-th moment and an explicit integral moment estimate for the sequence $(\null_{n}W)_{n\in\mathbb{N}}$ of adapted linear interpolations of $W$ from~\cite{ContKalininSupp}[Lemmas 19 and 17].

\begin{enumerate}[(i)]
\item Let $(\null_{n}X)_{n\in\mathbb{N}}$ be a sequence of $\mathbb{R}^{m}$-valued right-continuous processes for which there are $c_{0}\geq 0$ and $q > 0$ such that $E[|\null_{n}X_{s} - \null_{n}X_{t}|^{p}] \leq c_{0}|s-t|^{1+q}$ for all $n\in\mathbb{N}$, each $j\in\{0,\dots,k_{n}-1\}$ and every $s,t\in [t_{j,n},t_{j+1,n}]$. Then there is $c_{p,q} > 0$ such that
\begin{equation}\label{Auxiliary Estimate 1}
E[\|L_{n}(\null_{n}X) - \null_{n}X\|_{\infty}^{p}] \leq c_{p,q} c_{0}|\mathbb{T}_{n}|^{q}
\end{equation}
for all $n\in\mathbb{N}$. To be precise, $c_{p,q}=2^{p-1}(1 + k_{0,p,q})(T-r)$.

\item Let $Z$ be an $\mathbb{R}^{d}$-valued random vector satisfying $Z\sim\mathcal{N}(0,\mathbbm{I}_{d})$. Then the constant $\hat{w}_{p,q}:= E[|Z|^{pq}]c_{\mathbb{T}}^{pq}$ satisfies
\begin{equation}\label{Auxiliary Estimate 2}
E\bigg[\bigg(\int_{s}^{t}|\null_{n}\dot{W}_{u}|^{q}\,du\bigg)^{p}\bigg] \leq \hat{w}_{p,q}|\mathbb{T}_{n}|^{-pq/2}(t-s)^{p}
\end{equation}
for all $n\in\mathbb{N}$ and each $s,t\in [r,T]$ with $s\leq t$.
\end{enumerate}

Next, we let $p\geq 2$ and state a Burkholder-Davis-Ghundy inequality for stochastic integrals with respect to $W$ from~\cite{MaoSDEandApp}[Theorem 7.2]. Based on this bound, one can deduce an estimate for integrals relative to $\null_{n}W$ that is independent of $n\in\mathbb{N}$ and which is given in~\cite{ContKalininSupp}[Proposition 16].

\begin{enumerate}[(i)]
\setcounter{enumi}{2}
\item For each $\mathbb{R}^{m\times d}$-valued progressively measurable process $X$ for which $\int_{r}^{T}E[|X_{u}|^{p}]\,du$ is finite,
\begin{equation}\label{Auxiliary Estimate 3}
E\bigg[\sup_{v\in [s,t]}\bigg|\int_{s}^{v}X_{u}\,dW_{u}\bigg|^{p}\bigg] \leq w_{p}(t-s)^{p/2-1}\int_{s}^{t}E[|X_{u}|^{p}]\,du
\end{equation}
for  all $s,t\in [r,T]$ with $s\leq t$, where $w_{p}:=((p^{3}/2)/(p-1))^{p/2}$.

\item Any $\mathbb{R}^{m\times d}$-valued progressively measurable process $X$ satisfies
\begin{equation}\label{Auxiliary Estimate 4}
E\bigg[\max_{v\in [s,t]}\int_{s}^{v}X_{\underline{u}_{n}}\,d\null_{n}W_{u}\bigg|^{p}\bigg]\leq \hat{w}_{p}(t-s)^{p/2}\max_{j\in\{0,\dots,k_{n}\}:\,t_{j,n}\in [\underline{s}_{n},\underline{t}_{n}]} E[|X_{t_{j,n}}|^{p}]
\end{equation}
for each $s,t\in [r,T]$ with $s\leq t$ with $\hat{w}_{p}:=3^{p} w_{p}c_{\mathbb{T}}^{p/2}$.
\end{enumerate}

\subsection{Moment estimates for Volterra integrals}\label{Section 3.3}

The first integral bound that we consider follows from the auxiliary estimate~\eqref{Auxiliary Estimate 2}.

\begin{Lemma}\label{Auxiliary Lemma 1}
Let $p > 1$ and assume for each $n\in\mathbb{N}$ that $\null_{n}X:[0,T]\times [0,T]\times\Omega\rightarrow\mathbb{R}_{+}$, $(t,s,\omega)\mapsto X_{t,s}(\omega)$ is a product measurable function. If there are $\overline{p} > p$, $c_{\overline{p}} > 0$ and $q \geq \overline{p}/2$ such that
\begin{equation}\label{Auxiliary Estimation Condition}
E\bigg[\max_{j\in\{0,\dots,k_{n}\}}\int_{r}^{t_{j,n}}\null_{n}X_{t_{j,n},s}^{\overline{p}}\,ds\bigg] \leq c_{\overline{p}}|\mathbb{T}_{n}|^{q}\quad\text{for all  $n\in\mathbb{N}$.}
\end{equation}
Then there is $c_{p} > 0$ such that
\[
E\bigg[\max_{j\in\{0,\dots,k_{n}\}}\bigg(\int_{r}^{t_{j,n}}X_{t_{j,n},s}|\null_{n}\dot{W}_{s}|\,ds\bigg)^{p}\bigg]\leq c_{p}|\mathbb{T}_{n}|^{p(q/\overline{p} - 1/2)}\quad\text{for any $n\in\mathbb{N}$.}
\]
\end{Lemma}

\begin{proof}
Let $q_{1}$ and $q_{2}$ denote the dual exponents of $p$ and $\overline{p}/p$, respectively. Then two applications of H\oe lder's inequality yield that
\begin{align*}
E\bigg[\max_{j\in\{0,\dots,k_{n}\}}\bigg(\int_{r}^{t_{j,n}}&X_{t_{j,n},s}|\null_{n}\dot{W}_{s}|\,ds\bigg)^{p}\bigg]\\
&\leq \bigg(E\bigg[\max_{j\in\{0,\dots,k_{n}\}}\bigg(\int_{r}^{t_{j,n}}\null_{n}X_{t_{j,n},s}^{p}\,ds\bigg)^{\overline{p}/p}\bigg]\bigg)^{p/\overline{p}}c_{p,1}|\mathbb{T}_{n}|^{-p/2}
\end{align*}
with $c_{p,1}:= \hat{w}_{pq_{2}/q_{1},q_{1}}^{1/q_{2}}(T-r)^{p/q_{1}}$, where $\hat{w}_{pq_{2}/q_{1},q_{1}}$ is the constant introduced at~\eqref{Auxiliary Estimate 2}. For this reason, the constant $c_{p}:= (T-r)^{1 - p/\overline{p}}c_{\overline{p}}^{p/\overline{p}}c_{p,1}$ satisfies the desired estimate.
\end{proof}

\begin{Remark}\label{Auxiliary Remark}
For any $n\in\mathbb{N}$ let $\null_{n}X$ be independent of the first time variable, that is, there is an $\mathbb{R}_{+}$-valued measurable process $\null_{n}Y$ with $\null_{n}X_{t,s} = \null_{n}Y_{s}$ for all $s,t\in [0,T]$. Then for condition~\eqref{Auxiliary Estimation Condition} to hold, it suffices that there is $\overline{c}_{\overline{p}} > 0$ so that $E[\null_{n}Y_{s}^{\overline{p}}]\leq \overline{c}_{\overline{p}}|\mathbb{T}_{n}|^{q}$ for every $s\in [r,T)$ and each $n\in\mathbb{N}$.
\end{Remark}

For the second and various other estimates in the following section, let us use for each $n\in\mathbb{N}$ the function $\gamma_{n}:[r,T]\rightarrow [0,c_{\mathbb{T}}]$ defined by
\begin{equation}\label{Gamma Function}
\gamma_{n}(s) := \frac{\Delta s_{n}}{\Delta \overline{s}_{n}}.
\end{equation}
Put differently, $\gamma_{n} = \Delta t_{i,n}/\Delta t_{i+1,n}$ on $[t_{i,n},t_{i+1,n})$ for all $i\in\{0,\dots,k_{n}-1\}$ and $\gamma_{n}(T) = 1$.

\begin{Lemma}\label{Auxiliary Lemma 2}
Assume that $F:[r,T]^{2}\times C([0,T],\mathbb{R}^{m})\rightarrow\mathbb{R}^{m}$ is a non-anticipative product measurable map for which there are $\lambda_{0},c_{0}\geq 0$ such that
\[
|F(u,t,x) - F(u,s,x)|\leq \lambda_{0}d_{\infty}((t,x),(s,x))\quad\text{and}\quad |F(t,s,x)|\leq c_{0}(1 + \|x\|_{\infty})
\]
for all $s,t,u\in [r,T]$ with $s < t < u$ and each $x\in C([0,T],\mathbb{R}^{m})$. Further, let $(\null_{n}Y)_{n\in\mathbb{N}}$ be a sequence in $\mathscr{C}([0,T],\mathbb{R}^{m})$ which there are $p\geq 1$ and $c_{p,0}\geq 0$ such that
\[
E[\|\null_{n}Y\|_{\infty}^{p}] + E[\|\null_{n}Y^{s} - \null_{n}Y^{t}\|_{\infty}^{p}]/|s-t|^{p/2}\leq c_{p,0}\big(1 + E[\|\null_{n}Y^{r}\|_{\infty}^{p}]\big)
\]
for all $n\in\mathbb{N}$, each $s,t\in [r,T]$ with $s < t$ and any $x\in C([0,T],\mathbb{R}^{m})$. Then there is $c_{p} > 0$ such that
\[
E\bigg[\max_{j\in\{0,\dots,k_{n}\}}\bigg|\int_{r}^{t_{j,n}}F(t_{j,n},\underline{s}_{n},\null_{n}Y)\big(\gamma_{n}(s) - 1\big)\,ds\bigg|^{p}\bigg] \leq c_{p}|\mathbb{T}_{n}|^{p/2}\big(1 + E[\|\null_{n}Y^{r}\|_{\infty}^{p}]\big)
\]
for every $n\in\mathbb{N}$.
\end{Lemma}

\begin{proof}
Let $E[\|\null_{n}Y^{r}\|_{\infty}^{p}] < \infty$, as otherwise the claimed estimate is infinite. Clearly, a decomposition of the integral shows that
\[
\int_{r}^{t_{j,n}}F(t_{j,n},\underline{s}_{n},\null_{n}Y)\gamma_{n}(s)\,ds = \int_{r}^{t_{j-1,n}}F(t_{j,n},s_{n},\null_{n}Y)\,ds
\]
for all $j\in\{1,\dots,k_{n}\}$. Hence, a first estimation gives
\[
E\bigg[\max_{j\in\{1,\dots,k_{n}\}}\bigg|\int_{r}^{t_{j-1,n}}F(t_{j,n},s_{n},\null_{n}Y) - F(t_{j,n},\underline{s}_{n},\null_{n}Y)\,ds\bigg|^{p}\bigg]\leq c_{p,1}|\mathbb{T}_{n}|^{p/2}\big(1 + E[\|\null_{n}Y^{r}\|_{\infty}^{p}]\big)
\]
for $c_{p,1}:=2^{p-1}(T-r)^{p}\lambda_{0}^{p}(1+c_{p,0})$ and a second yields that
\[
E\bigg[\max_{j\in\{1,\dots,k_{n}\}}\bigg|\int_{t_{j-1,n}}^{t_{j,n}}F(t_{j,n},\underline{s}_{n},\null_{n}Y)\,ds\bigg|^{p}\bigg]\leq c_{p,2}|\mathbb{T}_{n}|^{p}\big(1 + E[\|\null_{n}Y^{r}\|_{\infty}^{p}]\big)
\]
with $c_{p,2}:=2^{p-1}c_{0}^{p}(1 + c_{p,0})$. Thus, the constant $c_{p}:=2^{p-1}(c_{p,1} + (T-r)^{p/2}c_{p,2})$ satisfies the asserted estimate.
\end{proof}

The third estimate deals with Volterra integrals driven by $\null_{n}W$ and $W$, where $n\in\mathbb{N}$.

\begin{Proposition}\label{Auxiliary Proposition}
Let $F:[r,T]^{2}\times C([0,T],\mathbb{R}^{m})\rightarrow\mathbb{R}^{m\times d}$ be non-anticipative, product measurable and such that $F(\cdot,s,x)$ is absolutely continuous on $[s,T]$ for all $s\in [r,T]$ and each $x\in C([0,T],\mathbb{R}^{m})$. Suppose that there are $\lambda_{0},c_{0}\geq 0$ such that
\begin{equation*}
\begin{split}
|F(u,t,x) - F(u,s,x)| + |\partial_{u}F(u,t,x) - \partial_{u}F(u,s,x)|&\leq \lambda_{0} d_{\infty}((t,x),(s,x))\\
|F(t,s,x)| + |\partial_{t}F(t,s,x)|&\leq c_{0}(1 + \|x\|_{\infty})
\end{split}
\end{equation*}
for any $s,t,u\in [r,T)$ with $s < t < u$ and every $x\in C([0,T],\mathbb{R}^{m})$. Moreover, let $(\null_{n}Y)_{n\in\mathbb{N}}$ be a sequence in $\mathscr{C}([0,T],\mathbb{R}^{m})$ for which there are $p\geq 2$ and $c_{p,0}\geq 0$ such that
\[
E[\|\null_{n}Y\|_{\infty}^{p}] + E[\|\null_{n}Y^{s} - \null_{n}Y^{t}\|_{\infty}^{p}]/|s-t|^{p/2} \leq c_{p,0}\big(1 + E[\|\null_{n}Y^{r}\|_{\infty}^{p}]\big)
\]
for all $n\in\mathbb{N}$, each $s,t\in [r,T]$ with $s < t$ and any $x\in C([0,T],\mathbb{R}^{m})$. Then there is $c_{p} > 0$ such that
\[
E\bigg[\max_{j\in\{0,\dots,k_{n}\}}\bigg|\int_{r}^{t_{j,n}} F(t_{j,n},\underline{s}_{n},\null_{n}Y)\,d(\null_{n}W_{s} - W_{s})\bigg|^{p}\bigg] \leq c_{p}|\mathbb{T}_{n}|^{p/2-1}\big(1 + E[\|\null_{n}Y^{r}\|_{\infty}^{p}]\big)
\]
for every $n\in\mathbb{N}$.
\end{Proposition}

\begin{proof}
We suppose that $E[\|\null_{n}Y^{r}\|_{\infty}^{p}]$ is finite and decompose the integral to get that
\[
\int_{r}^{t_{j,n}}F(t_{j,n},\underline{s}_{n},\null_{n}Y)\,d\null_{n}W_{s} = \int_{r}^{t_{j-1,n}}F(t_{j,n},s_{n},\null_{n}Y)\,dW_{s}\quad\text{a.s.}
\]
for each $j\in\{1,\dots,k_{n}\}$. Hence, we may apply Fubini's theorem for stochastic integrals from~\cite{VeraarFubini} to obtain that
\begin{align*}
\int_{r}^{t_{j-1,n}}&F(t_{j,n},s_{n},\null_{n}Y) - F(t_{j,n},\underline{s}_{n},\null_{n}Y)\,dW_{s} = \int_{r}^{t_{j-1,n}}F(s,s_{n},\null_{n}Y) - F(s,\underline{s}_{n},\null_{n}Y)\,dW_{s}\\
&+ \int_{r}^{t_{j,n}}\int_{r}^{t\wedge t_{j-1,n}}\partial_{t}F(t,s_{n},\null_{n}Y) - \partial_{t}F(t,\underline{s}_{n},\null_{n}Y)\,dW_{s}\,dt\quad\text{a.s.}
\end{align*}
for all $j\in\{1,\dots,k_{n}\}$. Regarding the first expression, we estimate that
\[
E\bigg[\max_{j\in\{1,\dots,k_{n}\}}\bigg|\int_{r}^{t_{j-1,n}}F(s,s_{n},\null_{n}Y)-F(s,\underline{s}_{n},\null_{n}Y)\,dW_{s}\bigg|^{p}\bigg]\leq c_{p,1}|\mathbb{T}_{n}|^{p/2}\big(1 + E[\|\null_{n}Y^{r}\|_{\infty}^{p}]\big)
\]
for $c_{p,1}:=2^{p-1}w_{p}(T-r)^{p/2}\lambda_{0}^{p}(1 + c_{p,0})$, where $w_{p}$ is the constant satisfying~\eqref{Auxiliary Estimate 3}. For the second expression we first calculate that
\begin{align*}
E\bigg[\max_{j\in\{1,\dots,k_{n}\}}\bigg|\int_{r}^{t_{j-1,n}}\int_{r}^{t}\partial_{t}F(t,s_{n},\null_{n}Y) - &\partial_{t}F(t,\underline{s}_{n},\null_{n}Y)\,dW_{s}\,dt\bigg|^{p}\bigg]\\
&\leq c_{p,2}|\mathbb{T}_{n}|^{p/2}\big(1 + E[\|\null_{n}Y^{r}\|_{\infty}^{p}]\big)
\end{align*}
with $c_{p,2}:=2^{p}(p+2)^{-1}w_{p}(T-r)^{3p/2}\lambda_{0}^{p}(1 + c_{p,0})$. And secondly,
\begin{align*}
&E\bigg[\max_{j\in\{1,\dots,k_{n}\}}\bigg|\int_{t_{j-1,n}}^{t_{j,n}}\int_{r}^{t_{j-1,n}}\partial_{t}F(t,s_{n},\null_{n}Y) - \partial_{t}F(t,\underline{s}_{n},\null_{n}Y)\,dW_{s}\,dt\bigg|^{p}\bigg]\\
&\leq \sum_{j=1}^{k_{n}}(t_{j,n} - t_{j-1,n})^{p-1}\int_{t_{j-1,n}}^{t_{j,n}}E\bigg[\bigg|\int_{r}^{t_{j-1,n}}\partial_{t}F(t,s_{n},\null_{n}Y) - \partial_{t}F(t,\underline{s}_{n},\null_{n}Y)\,dW_{s}\bigg|^{p}\bigg]\,dt\\
&\leq c_{p,3}|\mathbb{T}_{n}|^{p-1}\big(1 + E[\|\null_{n}Y^{r}\|_{\infty}^{p}]\big),
\end{align*}
where $c_{p,3}:=2^{p-1}w_{p}(T-r)^{p/2 + 1}\lambda_{0}^{p}(1 + c_{p,0})$. Next, for the remaining term Fubini's theorem for stochastic integrals yields that
\begin{equation}\label{Auxiliary Decomposition}
\begin{split}
\int_{t_{j-1,n}}^{t_{j,n}}F(t_{j,n},\underline{s}_{n},\null_{n}Y)\,dW_{s} &= \int_{t_{j-1,n}}^{t_{j,n}}F(s,\underline{s}_{n},\null_{n}Y)\,dW_{s}\\
&\quad + \int_{t_{j-1,n}}^{t_{j,n}}\int_{t_{j-1,n}}^{t}\partial_{t}F(t,\underline{s}_{n},\null_{n}Y)\,dW_{s}\,dt\quad\text{a.s.}
\end{split}
\end{equation}
for any $j\in\{1,\dots,k_{n}\}$. For the first term we have
\begin{align*}
E\bigg[\max_{j\in\{1,\dots,k_{n}\}}\bigg|\int_{t_{j-1,n}}^{t_{j,n}}F(s,\underline{s}_{n},\null_{n}Y)\,dW_{s}\bigg|^{p}\bigg]&\leq \sum_{j=1}^{k_{n}}E\bigg[\bigg|\int_{t_{j-1,n}}^{t_{j,n}}F(s,\underline{s}_{n},\null_{n}Y)\,dW_{s}\bigg|^{p}\bigg]\\
&\leq c_{p,4}|\mathbb{T}_{n}|^{p/2-1}\big(1 + E[\|\null_{n}Y^{r}\|_{\infty}^{p}]\big)
\end{align*}
with $c_{p,4}:=2^{p-1}w_{p}(T-r)c_{0}^{p}(1 + c_{p,0})$. Finally, for the second stochastic integral in the decomposition~\eqref{Auxiliary Decomposition} it holds that
\begin{align*}
E\bigg[\max_{j\in\{1,\dots,k_{n}\}}&\bigg|\int_{t_{j-1,n}}^{t_{j,n}}\int_{t_{j-1,n}}^{t}\partial_{t}F(t,\underline{s}_{n},\null_{n}Y)\,dW_{s}\,dt\bigg|^{p}\bigg]\\
&\leq \sum_{j=1}^{k_{n}}(t_{j,n} - t_{j-1,n})^{p-1}\int_{t_{j-1,n}}^{t_{j,n}} E\bigg[\bigg|\int_{t_{j-1,n}}^{t}\partial_{t}F(t,\underline{s}_{n},\null_{n}Y)\,dW_{s}\bigg|^{p}\bigg]\,dt\\
&\leq c_{p,5}|\mathbb{T}_{n}|^{3p/2-1}\big(1 + E[\|\null_{n}Y^{r}\|_{\infty}^{p}]\big)
\end{align*}
for $c_{p,5}:= 2^{p}(p+2)^{-1}w_{p}(T-r)c_{0}^{p}(1 + c_{p,0})$. Hence, the asserted estimate follows readily by setting $c_{p}:=5^{p-1}((T-r)(c_{p,1} + c_{p,2}) + (T-r)^{p/2}c_{p,3} + c_{p,4} + (T-r)^{p}c_{p,5})$.
\end{proof}

\section{Estimates and decompositions for the convergence result}\label{Section 4}

\subsection{Decomposition into remainder terms}\label{Section 4.1}

We first give a moment estimate for solutions to~\eqref{Sequence VIE} that does not depend on $n\in\mathbb{N}$.

\begin{Proposition}\label{Sequence L-p Estimate Proposition}
Let~~\eqref{C.5} and~\eqref{C.6} hold, $h\in W_{r}^{1,2}([0,T],\mathbb{R}^{d})$ and $\lambda\geq 0$ be so that
\[
|\overline{B}(u,t,x) - \overline{B}(u,s,x)| + |\partial_{u}\overline{B}(u,t,x) - \partial_{u}\overline{B}(u,s,x)|\leq \lambda d_{\infty}((t,x),(s,x))
\]
for any $s,t,u\in [r,T)$ with $s < t < u$ and every $x\in C([0,T],\mathbb{R}^{m})$. Then for each $p\geq 2$ there is $c_{p} > 0$ such that any $n\in\mathbb{N}$ and each solution $\null_{n}Y$ to~\eqref{Sequence VIE} satisfy
\begin{equation}\label{Sequence L-p Estimate}
E[\|\null_{n}Y\|_{\infty}^{p}] + E[\|\null_{n}Y^{s} - \null_{n}Y^{t}\|_{\infty}^{p}]/|s-t|^{p/2}\leq c_{p}\big(1 + E[\|\null_{n}Y^{r}\|_{\infty}^{p}]\big)
\end{equation}
for all $s,t\in [r,T]$ with $s \neq t$.
\end{Proposition}

\begin{proof}
We let $E[\|\null_{n}Y^{r}\|_{\infty}^{p}] < \infty$ and may certainly assume in~\eqref{C.6} that $\kappa > 0$. For given $l\in\mathbb{N}$ the stopping time $\tau_{l,n}:=\inf\{t\in [0,T]\,|\, |\null_{n}Y_{t}|\geq l\}\vee r$ satisfies $\|\null_{n}Y^{\tau_{l,n}}\|_{\infty}\leq \|\null_{n}Y^{r}\|_{\infty}\vee l$ and we readily estimate that
\begin{equation}\label{Sequence L-p Estimate Equation}
\begin{split}
\big(E[\|\null_{n}Y^{s\wedge\tau_{l,n}} - &\null_{n}Y^{t\wedge\tau_{l,n}}\|_{\infty}^{p}]\big)^{1/p}\leq \bigg(\overline{c}_{p}(t-s)^{p/2-1}\int_{s}^{t}1 + E[\|\null_{n}Y^{u\wedge\tau_{l,n}}\|_{\infty}^{\kappa p}]\,du\bigg)^{1/p}\\
& + \bigg(E\bigg[\sup_{v\in [s,t]}\bigg|\int_{s}^{v\wedge\tau_{l,n}}\overline{B}(u,u,\null_{n}Y)\,d\null_{n}W_{u}\bigg|^{p}\bigg]\bigg)^{1/p}\\
& + \bigg(E\bigg[\bigg(\int_{s}^{t\wedge\tau_{l,n}}\bigg|\int_{r}^{v}\partial_{v}\overline{B}(v,u,\null_{n}Y)\,d\null_{n}W_{u}\bigg|\,dv\bigg)^{p}\bigg]\bigg)^{1/p}
\end{split}
\end{equation}
for any fixed $s,t\in [r,T]$ with $s \leq t$ and $\overline{c}_{p}:=6^{p-1}(1 + T - r)^{p}((T-r)^{p/2} + \|h\|_{1,2,r}^{p} + w_{p})c^{p}$. We recall the constant $\hat{w}_{p/\kappa,1}$ such that~\eqref{Auxiliary Estimate 2} holds when $p$ and $q$ are replaced by $p/\kappa$ and $1$, respectively. Then
\begin{align*}
\bigg(E\bigg[\bigg(\int_{\underline{u}_{n}}^{u\wedge\tau_{l,n}}|\overline{B}(v,v,\null_{n}Y)\null_{n}\dot{W}_{v}|\,dv\bigg)^{p/\kappa}\bigg]\bigg)^{\kappa}&\leq c_{p,1}(u-\underline{u}_{n})^{p/2}\quad\text{and}\\
\bigg(E\bigg[\bigg(\int_{\underline{u}_{n}}^{u\wedge\tau_{l,n}}\int_{r}^{v}|\partial_{v}\overline{B}(v,u',\null_{n}Y)\null_{n}\dot{W}_{u'}|\,du'\,dv\bigg)^{p/\kappa}\bigg]\bigg)^{\kappa}&\leq (T-r)^{p}c_{p,1}(u-\underline{u}_{n})^{p/2}
\end{align*}
for any given $u\in [s,T]$ with the constant$c_{p,1}:=2^{p/2}\hat{w}_{p/\kappa,1}^{\kappa}c^{p}$. We let $\overline{c}_{p/\kappa}$ be defined just as $\overline{c}_{p}$ above with $p$ replaced by $p/\kappa$ to get that
\[
(E[\|\null_{n}Y^{u\wedge\tau_{l,n}} - \null_{n}Y^{\underline{u}_{n}\wedge\tau_{l,n}}\|_{\infty}^{p/\kappa}])^{\kappa} \leq c_{p,2}(u-\underline{u}_{n})^{p/2}\big(1 + E[\|\null_{n}Y^{u\wedge\tau_{l,n}}\|_{\infty}^{p}]\big)^{\kappa}
\]
for $c_{p,2}:=2^{p-1}(\overline{c}_{p/\kappa}^{\kappa} + (1 + T -r)^{p}c_{p,1})$, due to the validity of~\eqref{Sequence L-p Estimate Equation}. Hence, an application of H\oe lder's inequality yields that
\begin{align*}
&E\bigg[\bigg(\int_{s}^{t\wedge\tau_{l,n}}\big|\big(\overline{B}(u,u,\null_{n}Y) - \overline{B}(\underline{u}_{n},\underline{u}_{n},\null_{n}Y)\,\null_{n}\dot{W}_{u}\big|\,du\bigg)^{p}\bigg]\\
&\leq c_{p,3}(t-s)^{p/2-1}\int_{s}^{t}\big(1 + E[\|\null_{n}Y^{u\wedge\tau_{l,n}}\|_{\infty}^{p}]\big)^{\kappa}\,du\quad\text{and}\\
&E\bigg[\bigg(\int_{s}^{t\wedge\tau_{l,n}}\int_{r}^{v}\big|\big(\partial_{v}\overline{B}(v,u,\null_{n}Y) - \partial_{v}\overline{B}(v,\underline{u}_{n},\null_{n}Y)\big)\null_{n}\dot{W}_{u}\big|\,du\,dv\bigg)^{p}\bigg]\\
&\leq (T-r)^{p}c_{p,3}(t-s)^{p/2-1}\int_{s}^{t}\big(1 + E[\|\null_{n}Y^{u\wedge\tau_{l,n}}\|_{\infty}^{p}]\big)^{\kappa}\,du,
\end{align*}
where $c_{p,3}:=2^{p/2}3^{p}\hat{w}_{(p/2)/(1-\kappa)}^{1-\kappa}(\lambda^{p}(1 + c_{p,2}) + c^{p}(T-r)^{p/2})$. Moreover, the constant $\hat{w}_{p}$ appearing in~\eqref{Auxiliary Estimate 4} satisfies 
\begin{align*}
&E\bigg[\sup_{v\in [s,t]}\bigg|\int_{s}^{v\wedge\tau_{l,n}}\overline{B}(\underline{u}_{n},\underline{u}_{n},\null_{n}Y)\,d\null_{n}W_{u}\bigg|^{p}\bigg] \leq \hat{w}_{p}c^{p}(t-s)^{p/2}\quad\text{and}\\
&E\bigg[\bigg(\int_{s}^{t\wedge\tau_{l,n}}\bigg|\int_{r}^{v}\partial_{v}\overline{B}(v,\underline{u}_{n},\null_{n}Y)\,d\null_{n}W_{u}\bigg|\,dv\bigg)^{p}\bigg]\leq (T-r)^{p}\hat{w}_{p}c^{p}(t-s)^{p/2}.
\end{align*}
Thus, with the constant $c_{p,4}:=3^{p-1}(2\overline{c}_{p} + (1+T-r)^{p}(c_{p,3} + \hat{w}_{p}c^{p})$ we can now infer from~\eqref{Sequence L-p Estimate Equation} that
\begin{equation}\label{Sequence L-p Estimate Equation 2}
E[\|\null_{n}Y^{s\wedge\tau_{l,n}} - \null_{n}Y^{t\wedge\tau_{l,n}}\|_{\infty}^{p}\big] \leq c_{p,4}(t-s)^{p/2-1}\int_{s}^{t}1 + E[\|\null_{n}Y^{u\wedge\tau_{l,n}}\|_{\infty}^{p}]\,du.
\end{equation}
Hence, Gronwall's inequality and Fatou's lemma imply that
\begin{equation*}
E[\|\null_{n}Y^{t}\|_{\infty}^{p}]\leq\liminf_{l\uparrow\infty} E[\|\null_{n}Y^{t\wedge\tau_{l,n}}\|_{\infty}^{p}]\leq c_{p,5}\big(1 + E[\|\null_{n}Y^{r}\|_{\infty}^{p}]\big),
\end{equation*}
where $c_{p,5}:=2^{p-1}\max\{1,T-r\}^{p/2}\max\{1,c_{p,4}\}e^{2^{p-1}(T-r)^{p/2}c_{p,4}}$. For this reason, we set $c_{p}:=(1 + c_{p,4})(1 + c_{p,5})$ and apply Fatou's lemma to~\eqref{Sequence L-p Estimate Equation 2}, which gives the result.
\end{proof}

\begin{Corollary}\label{Sequence L-p Estimate Corollary}
Assume~\eqref{C.5},~\eqref{C.6} and~\eqref{C.8} and let $h\in W_{r}^{1,2}([0,T],\mathbb{R}^{d})$. Then for every $p\geq 2$ there is $c_{p} > 0$ such that each solution $Y$ to~\eqref{General VIE} satisfies
\begin{equation}\label{Sequence L-p Estimate 2}
E[\|Y\|_{\infty}^{p}] + E[\|Y^{s} - Y^{t}\|_{\infty}^{p}]/|s-t|^{p/2} \leq c_{p}\big(1 + E[\|Y^{r}\|_{\infty}^{p}]\big)
\end{equation}
for every $s,t\in [r,T]$ with $s\neq t$.
\end{Corollary}

\begin{proof}
As the map $R$ given by~\eqref{General Remainder Definition} is bounded, the assertion is a direct consequence of Proposition~\ref{Sequence L-p Estimate Proposition} by replacing $\underline{B}$ by $\underline{B} + R$, $\overline{B}$ by $0$ and $\Sigma$ by $\overline{B} + \Sigma$.
\end{proof}

For $n\in\mathbb{N}$ let us recall the linear operator $L_{n}$ and the function $\gamma_{n}$ given at~\eqref{Linear Operator L} and~\eqref{Gamma Function}, respectively, and deduce the main decomposition to establish the limit~\eqref{General Partition Limit}.

\begin{Proposition}\label{Main Decomposition Proposition}
Let~\eqref{C.5}-\eqref{C.8} hold and $h\in W_{r}^{1,2}([0,T],\mathbb{R}^{d})$. Then for each $p\geq 2$ there is $c_{p} > 0$ such that each $n\in\mathbb{N}$ and any two solutions $\null_{n}Y$ and $Y$ of~\eqref{Sequence VIE} and~\eqref{General VIE}, respectively, satisfy
\begin{align*}
&E\big[\max_{j\in\{0,\dots,k_{n}\}} |\null_{n}Y_{t_{j,n}} - Y_{t_{j,n}}|^{p}\big]/c_{p}\leq |\mathbb{T}_{n}|^{p/2}\big(1 + E\big[\|\null_{n}Y^{r}\|_{\infty}^{p} + \|Y^{r}\|_{\infty}^{p}]\big)\\
&+ E\big[\|\null_{n}Y^{r} - Y^{r}\|_{\infty}^{p} + \|L_{n}(\null_{n}Y)-\null_{n}Y\|_{\infty}^{p} + \|L_{n}(Y) - Y\|_{\infty}^{p}\big]\\
& + E\bigg[\max_{j\in\{0,\dots,k_{n}\}}\bigg|\int_{r}^{t_{j,n}}R(t_{j,n},\underline{s}_{n},\null_{n}Y)\big(\gamma_{n}(s) - 1\big)\,ds\bigg|^{p}\bigg]\\
& + E\bigg[\max_{j\in\{0,\dots,k_{n}\}}\bigg|\int_{r}^{t_{j,n}}\overline{B}(t_{j,n},\underline{s}_{n},\null_{n}Y)\,d(\null_{n}W_{s} - W_{s})\bigg|^{p}\bigg]\\
& + E\bigg[\max_{j\in\{0,\dots,k_{n}\}}\bigg|\int_{r}^{t_{j,n}}\big(\overline{B}(t_{j,n},s,\null_{n}Y)-\overline{B}(t_{j,n},\underline{s}_{n},\null_{n}Y)\big)\null_{n}\dot{W}_{s} - R(t_{j,n},\underline{s}_{n},\null_{n}Y)\gamma_{n}(s)\,ds\bigg|^{p}\bigg].
\end{align*}
\end{Proposition}

\begin{proof}
We suppose that $E[\|\null_{n}Y^{r}\|_{\infty}^{p}]$ and $E[\|Y^{r}\|_{\infty}^{p}]$ are finite and aim to derive the estimate by applying Gronwall's inequality to the increasing function $\varphi_{n}:[r,T]\rightarrow\mathbb{R}_{+}$ given by
\[
\varphi_{n}(t):=E\big[\max_{j\in\{0,\dots,k_{n}\}:\,t_{j,n}\leq t} |\null_{n}Y_{t_{j,n}} - Y_{t_{j,n}}|^{p}\big].
\]
To this end, let us write the difference of $\null_{n}Y$ and $Y$ as follows:
\begin{align*}
\null_{n}Y_{t} - Y_{t} &= \null_{n}Y_{r} - Y_{r} + \int_{r}^{t}\underline{B}(s,s,\null_{n}Y) - \underline{B}(s,s,Y)\,ds\\
&\quad + \int_{r}^{t}B_{H}(s,s,\null_{n}Y) - B_{H}(s,s,Y)\,dh(s)\\
&\quad + \null_{n}\Delta_{t} + \int_{r}^{t}\int_{r}^{v}\partial_{v}\underline{B}(v,u,\null_{n}Y)-\partial_{v}\underline{B}(v,u,Y)\,du\,dv\\
&\quad + \int_{r}^{t}\int_{r}^{v}\partial_{v}B_{H}(v,u,\null_{n}Y)-\partial_{v}B_{H}(v,u,Y)\,dh(u)\,dv\\
&\quad + \int_{r}^{t}\Sigma(s,s,\null_{n}Y) - \Sigma(s,s,Y)\,dW_{s}\\
&\quad + \int_{r}^{t}\int_{r}^{v}\partial_{v}\Sigma(v,u,\null_{n}Y) -\partial_{v}\Sigma(v,u,Y)\,dW_{u}\,dv
\end{align*}
for each $t\in [r,T]$ a.s. with a process $\null_{n}\Delta\in\mathscr{C}([0,T],\mathbb{R}^{m})$ satisfying
\[
\null_{n}\Delta_{t} = \int_{r}^{t}\overline{B}(t,s,\null_{n}Y)\null_{n}\dot{W}_{s} - R(t,s,Y)\,ds - \int_{r}^{t}\overline{B}(t,s,Y)\,dW_{s}
\]
for any $t\in [r,T]$ a.s. So, we let the  terms $\null_{n}Y_{r} - Y_{r}$ and $\null_{n}\Delta$ unchanged, then for the constant $c_{p,1}:=15^{p-1}(1+T-r)^{p}(T-r)^{p/2-1}((T-r)^{p/2} + \|h\|_{1,2,r}^{p} + w_{p})\lambda^{p}$ we have
\begin{equation}\label{Main Remainder Estimate 1}
\varphi_{n}(t)^{1/p} \leq \delta_{n,1}^{1/p} + \delta_{n}(t)^{1/p} + \bigg(c_{p,1}\int_{r}^{t_{n}}\delta_{n,1} + \delta_{n,2}(s) + \varepsilon_{n}(s) + \varphi_{n}(s)\,ds\bigg)^{1/p} 
\end{equation}
for all $t\in [r,T]$, where we have set $\delta_{n,1}:=E[\|\null_{n}Y^{r} - Y^{r}\|_{\infty}^{p}]$ and the measurable functions $\delta_{n},\delta_{n,2},\varepsilon_{n}:[r,T]\rightarrow\mathbb{R}_{+}$ are defined by
\begin{align*}
\delta_{n}(t)&:=E\big[\max_{j\in\{0,\dots,k_{n}\}:\,t_{j,n}\leq t} |\null_{n}\Delta_{t_{j,n}}|^{p}\big],\\
\delta_{n,2}(s)&:=E\big[\|L_{n}(\null_{n}Y)^{\underline{s}_{n}} - \null_{n}Y^{\underline{s}_{n}}\|_{\infty}^{p} + \|L_{n}(Y)^{\underline{s}_{n}} - Y^{\underline{s}_{n}}\|_{\infty}^{p}\big]\quad\text{and}\\
\varepsilon_{n}(s)&:= E\big[\|\null_{n}Y^{s} - \null_{n}Y^{\underline{s}_{n}}\|_{\infty}^{p} + \|Y^{s} - Y^{\underline{s}_{n}}\|_{\infty}^{p}\big].
\end{align*}
To obtain the estimate~\eqref{Main Remainder Estimate 1}, we used the chain of inequalities: $E[\|L_{n}(\null_{n}Y)^{\underline{s}_{n}} - L_{n}(Y)^{\underline{s}_{n}}\|_{\infty}^{p}]$ $\leq E[\|\null_{n}Y^{r} - Y^{r}\|_{\infty}^{p}\vee \max_{j\in\{0,\dots,k_{n}\}:\,t_{j,n}\leq s} |\null_{n}Y_{t_{j,n}} - Y_{t_{j,n}}|^{p}] \leq \delta_{n,1} + \varphi_{n}(s)$, valid for every $s\in [r,T]$.

For the estimation of $\delta_{n}$ let us define two processes $\null_{n,3}\Delta, \null_{n,5}\Delta\in\mathscr{C}([0,T],\mathbb{R}^{m})$ by $\null_{n,3}\Delta_{t} := \int_{r}^{t}R(t,\underline{s}_{n},\null_{n}Y)\big(\gamma_{n}(s) -1\big)\,ds$ and
\[
\null_{n,5}\Delta_{t}:=\int_{r}^{t}\big(\overline{B}(t,s,\null_{n}Y) - \overline{B}(t,\underline{s}_{n},\null_{n}Y)\big)\null_{n}\dot{W}_{s} - R(t,\underline{s}_{n},\null_{n}Y)\gamma_{n}(s)\,ds
\]
and choose $\null_{n,4}\Delta\in\mathscr{C}([0,T],\mathbb{R}^{m})$ such that $\null_{n,4}\Delta_{t} = \int_{r}^{t}\overline{B}(t,\underline{s}_{n},\null_{n}Y)\,d(\null_{n}W_{s} - W_{s})$ for any $t\in [r,T]$ a.s. Then $\null_{n}\Delta$ admits the following representation:
\begin{align*}
\null_{n}\Delta_{t} &= \null_{n,3}\Delta_{t} + \null_{n,4}\Delta_{t} + \null_{n,5}\Delta_{t} + \int_{r}^{t}R(t,\underline{s}_{n},\null_{n}Y) - R(t,s,Y)\,ds\\
&\quad + \int_{r}^{t}\overline{B}(s,\underline{s}_{n},\null_{n}Y) - \overline{B}(s,s,Y)\,dW_{s} + \int_{r}^{t}\int_{r}^{u}\partial_{u}\overline{B}(u,\underline{s}_{n},\null_{n}Y) - \partial_{u}\overline{B}(u,s,Y)\,dW_{s}\,du
\end{align*}
for all $t\in [r,T]$ a.s. Due to the assumptions, we may assume without loss of generality that the Lipschitz constant $\lambda$ is large enough such that
\begin{equation*}
|R(u,t,x) - R(u,s,y)|\leq \lambda d_{\infty}((t,x),(s,y))
\end{equation*}
for any $s,t,u\in [r,T)$ with $s < t < u$ and every $x,y\in C([0,T],\mathbb{R}^{m})$. Thus, for the constant $c_{p,2}:=10^{p-1}(1 + T-r)^{p}(T-r)^{p/2-1}((T-r)^{p/2} + w_{p})\lambda^{p}$ we get that
\begin{equation}\label{Main Remainder Estimate 2}
\begin{split}
\delta_{n}(t)^{1/p}&\leq \delta_{n,3}(t)^{1/p} + \delta_{n,4}(t)^{1/p} + \delta_{n,5}(t)^{1/p}\\
&\quad + \bigg(c_{p,2}\int_{r}^{t_{n}}\delta_{n,1} + (s-\underline{s}_{n})^{p/2} + \delta_{n,2}(s) + \varepsilon_{n}(s) + \varphi_{n}(s)\,ds\bigg)^{1/p}
\end{split}
\end{equation}
for every $t\in [r,T]$, where the increasing function $\delta_{n,i}:[r,T]\rightarrow\mathbb{R}_{+}$ is given through
\begin{equation*}
\delta_{n,i}(t):= E\big[\max_{j\in\{0,\dots,k_{n}\}:\,t_{j,n}\leq t} |\null_{n,i}\Delta_{t_{j,n}}|^{p}\big]\quad\text{for all $i\in\{3,4,5\}$.}
\end{equation*}

Thanks to Proposition~\ref{Sequence L-p Estimate Proposition} and Corollary~\ref{Sequence L-p Estimate Corollary}, there are $\underline{c}_{p},\overline{c}_{p} > 0$ such that~\eqref{Sequence L-p Estimate} and~\eqref{Sequence L-p Estimate 2} hold when $c_{p}$ is replaced by $\underline{c}_{p}$ and $\overline{c}_{p}$, respectively. By combining~\eqref{Main Remainder Estimate 1} with~\eqref{Main Remainder Estimate 2}, we see that
\begin{align*}
\varphi_{n}(t)&\leq c_{p,4}|\mathbb{T}_{n}|^{p/2}\big(1 + E[\|\null_{n}Y^{r}\|_{\infty}^{p} + \|Y^{r}\|_{\infty}^{p}\big]\big) +(5^{p-1} + c_{p,3}(T-r))\delta_{n,1}\\
&\quad + 5^{p-1}\big(\delta_{n,3}(t) + \delta_{n,4}(t) + \delta_{n,5}(t)\big) + c_{p,3}\int_{r}^{t_{n}}\delta_{n,2}(s) + \varphi_{n}(s)\,ds
\end{align*}
for fixed $t\in [r,T]$, where $c_{p,3}:=10^{p-1}(c_{p,1} + c_{p,2})$ and $c_{p,4}:=2^{p/2}(T-r)(1 + \underline{c}_{p} + \overline{c}_{p})c_{p,3}$. For this reason, Gronwall's inequality gives
\[
\varphi_{n}(t)/c_{p}\leq |\mathbb{T}_n|^{p/2}\big(1 + E[\|\null_{n}Y^{r}\|_{\infty}^{p} + \|Y^{r}\|_{\infty}^{p}]\big) + \delta_{n,1} + \sum_{i=2}^{5}\delta_{n,i}(t)
\]
with $c_{p}:=e^{c_{p,3}(T-r)}(5^{p-1} + c_{p,4})$, which implies the desired estimate.
\end{proof}

By the estimate~\eqref{Auxiliary Estimate 1}, Lemma~\ref{Auxiliary Lemma 2} and Proposition~\ref{Auxiliary Proposition}, to prove~\eqref{General Partition Limit}, only the last remainder in the estimation of Proposition~\ref{Main Decomposition Proposition} should be investigated in more detail. Thus, let $\Phi_{h,n}:[r,T]\times C([0,T],\mathbb{R}^{m})\times C([0,T],\mathbb{R}^{d})\rightarrow\mathbb{R}^{m}$ be defined via
\begin{align*}
\Phi_{h,n}(s,y,w)&:=B_{H}(\underline{s}_{n},\underline{s}_{n},y)(h(s) - h(\underline{s}_{n})) + \overline{B}(\underline{s}_{n},\underline{s}_{n},y)\big(L_{n}(w)(s) - L_{n}(w)(\underline{s}_{n})\big)\\
&\quad + \Sigma(\underline{s}_{n},\underline{s}_{n},y)(w(s) - w(\underline{s}_{n})) + \int_{\underline{s}_{n}}^{s}\int_{r}^{\underline{s}_{n}}\partial_{v}\overline{B}(v,u,y)\,dL_{n}(w)(u)\,dv
\end{align*}
for each $h\in W_{r}^{1,2}([0,T],\mathbb{R}^{d})$ and any $n\in\mathbb{N}$. Whenever $\null_{n}Y$ is a solution to~\eqref{Sequence VIE}, then we will utilize the following decomposition to deal with the considered remainder:
\begin{equation}\label{Main Remainder Decomposition}
\begin{split}
&\big(\overline{B}(t_{j,n},s,\null_{n}Y) - \overline{B}(t_{j,n},\underline{s}_{n},\null_{n}Y)\big)\null_{n}\dot{W}_{s} - R(t_{j,n},\underline{s}_{n},\null_{n}Y)\gamma_{n}(s)\\
&= \big(\overline{B}(t_{j,n},s,\null_{n}Y) - \overline{B}(t_{j,n},\underline{s}_{n},\null_{n}Y) - \partial_{x}\overline{B}(t_{j,n},\underline{s}_{n},\null_{n}Y)(\null_{n}Y_{s} - \null_{n}Y_{\underline{s}_{n}})\big)\null_{n}\dot{W}_{s}\\
&\quad + \partial_{x}\overline{B}(t_{j,n},\underline{s}_{n},\null_{n}Y)\big(\null_{n}Y_{s} - \null_{n}Y_{\underline{s}_{n}} - \Phi_{h,n}(s,\null_{n}Y,W)\big)\null_{n}\dot{W}_{s}\\
&\quad + \partial_{x}\overline{B}(t_{j,n},\underline{s}_{n},\null_{n}Y)\Phi_{h,n}(s,\null_{n}Y,W)\null_{n}\dot{W}_{s} - R(t_{j,n},\underline{s}_{n},\null_{n}Y)\gamma_{n}(s)
\end{split}
\end{equation}
for all $j\in\{1,\dots,k_{n}\}$ and each $s\in [r,t_{j,n})$.

\subsection{Moment estimates for the first two remainders}\label{Section 4.2}

The first result in this section together with Lemma~\ref{Auxiliary Lemma 1} provide an estimate of the first remainder appearing in~\eqref{Main Remainder Decomposition}.

\begin{Proposition}\label{Remainder Proposition 1}
Let \eqref{C.4}-\eqref{C.6} be valid, $h\in W_{r}^{1,2}([0,T],\mathbb{R}^{d})$ and $F$ be a product measurable functional on $[r,T]\times [r,T)\times C([0,T],\mathbb{R}^{m})$ so that the following two conditions hold:
\begin{enumerate}[(i)]
\item There exists $\lambda\geq 0$ such that $|\overline{B}(u,t,x) - \overline{B}(u,s,x)| + |\partial_{u}\overline{B}(u,t,x) - \partial_{u}\overline{B}(u,s,x)|$ $\leq\lambda d_{\infty}((t,x),(s,x))$ for any $s,t,u\in [r,T)$ with $s < t < u$ and all $x\in C([0,T],\mathbb{R}^{m})$.

\item The functional $[r,t)\times C([0,T],\mathbb{R}^{m})\rightarrow\mathbb{R}$, $(s,x)\mapsto F(t,s,x)$ is of class $\mathbb{C}^{1,2}$ for any $t\in (r,T]$ and there are $c_{0},\eta,\lambda_{0}\geq 0$ such that
\begin{align*}
|\partial_{s}F(t,s,x)| + |\partial_{xx}F(t,s,x)|&\leq c_{0}(1 + \|x\|_{\infty}^{\eta}),\\
|\partial_{x}F(u,t,x) - \partial_{x}F(u,s,x)|&\leq \lambda_{0} d_{\infty}((t,x),(s,x))
\end{align*}
for each $s,t,u\in [r,T)$ with $s < t < u$ and all $x\in C([0,T],\mathbb{R}^{m})$.
\end{enumerate}
Then for any $p\geq 2$ there is $c_{p} > 0$ such that for all $n\in\mathbb{N}$ and each solution $\null_{n}Y$ to~\eqref{Sequence VIE},
\begin{align*}
E\bigg[\max_{j\in\{1,\dots,k_{n}\}}\int_{r}^{t_{j,n}}\big|F(t_{j,n},s,\null_{n}Y) &- F(t_{j,n},\underline{s}_{n},\null_{n}Y) - \partial_{x}F(t_{j,n},\underline{s}_{n},\null_{n}Y)(\null_{n}Y_{s} - \null_{n}Y_{\underline{s}_{n}})\big|^{p}\,ds\bigg]\\
&\leq c_{p}|\mathbb{T}_{n}|^{p-1}\big(1 + E\big[\|\null_{n}Y^{r}\|_{\infty}^{(\eta\vee 2) p}\big]\big).
\end{align*}
\end{Proposition}

\begin{proof}
For any $j\in\{1,\dots,k_{n}\}$ let the product measurable map $\null_{n,j}\Delta:[r,t_{j,n})^{2}\times\Omega\rightarrow\mathbb{R}^{1\times m}$ be given by $\null_{n,j}\Delta_{s,u}:=\partial_{x}F(t_{j,n},u,\null_{n}Y) - \partial_{x}F(t_{j,n},\underline{s}_{n},\null_{n}Y)$, if $u\in [\underline{s}_{n},s]$, and $\null_{n,j}\Delta_{s,u}:=0$, otherwise. Then from the functional It{\^o} formula in~\cite{ContFournieFuncIto} we infer that
\begin{equation}\label{Ito Formula Decomposition}
\begin{split}
&F(t_{j,n},s,\null_{n}Y) - F(t_{j,n},\underline{s}_{n},\null_{n}Y) - \partial_{x}F(t_{j,n},\underline{s}_{n},\null_{n}Y)(\null_{n}Y_{s} - \null_{n}Y_{\underline{s}_{n}})\\
&= \int_{\underline{s}_{n}}^{s}\partial_{u}F(t_{j,n},u,\null_{n}Y) + \frac{1}{2}\mathrm{tr}(\partial_{xx}F(t_{j,n},u,\null_{n}Y)(\Sigma\Sigma')(u,u,\null_{n}Y))\,du\\
&\quad + \int_{\underline{s}_{n}}^{s}\null_{n,j}\Delta_{s,u}\big(\underline{B}(u,u,\null_{n}Y) + B_{H}(u,u,\null_{n}Y)\dot{h}(u) + \overline{B}(u,u,\null_{n}Y)\null_{n}\dot{W}_{u}\big)\,du\\
& + \int_{\underline{s}_{n}}^{s}\null_{n,j}\Delta_{s,v}\int_{r}^{v}\partial_{v}\underline{B}(v,u,\null_{n}Y) + \partial_{v}B_{H}(v,u,\null_{n}Y)\dot{h}(u) + \partial_{v}\overline{B}(v,u,\null_{n}Y)\null_{n}\dot{W}_{u}\,du\,dv\\
& + \int_{\underline{s}_{n}}^{s}\null_{n,j}\Delta_{s,u}\Sigma(u,u,\null_{n}Y)\,dW_{u} + \int_{\underline{s}_{n}}^{s}\null_{n,j}\Delta_{s,v}\int_{r}^{v}\partial_{v}\Sigma(v,u,\null_{n}Y)\,dW_{u}\,dv
\end{split}
\end{equation}
for each $s\in [r,t_{j,n})$ a.s. Now, for $\overline{\eta}:=\eta\vee 2$ Proposition~\ref{Sequence L-p Estimate Proposition} gives a constant $\underline{c}_{\overline{\eta}p} > 0$ such that~\eqref{Sequence L-p Estimate} holds when $p$ and $c_{p}$ are replaced by $\overline{\eta}p$ and $\underline{c}_{\overline{\eta}p}$, respectively. Then for the first two terms on the right-hand side in~\eqref{Ito Formula Decomposition} we have
\begin{align*}
&E\bigg[\max_{j\in\{1,\dots,k_{n}\}}\sup_{s\in [r,t_{j,n})}\bigg|\int_{\underline{s}_{n}}^{s}\partial_{u}F(t_{j,n},u,\null_{n}Y) + \frac{1}{2}\mathrm{tr}(\partial_{xx}F(t_{j,n},u,\null_{n}Y)(\Sigma\Sigma')(u,u,\null_{n}Y))\,du\bigg|^{p}\bigg]\\
&\leq 2^{p-1}c_{0}^{p}(s-\underline{s}_{n})^{p}E[(1 + \|\null_{n}Y\|_{\infty}^{\eta})^{p}]\\
&\quad + 2^{-1}c_{0}^{p}(s-\underline{s}_{n})^{p-1}\int_{\underline{s}_{n}}^{s}E\big[(1 + \|\null_{n}Y^{u}\|_{\infty}^{\eta})^{p}|(\Sigma\Sigma')(u,u,\null_{n}Y)|^{p}\big]\,du\\
&\leq c_{p,1}|\mathbb{T}_{n}|^{p}\big(1 + E\big[\|\null_{n}Y^{r}\|_{\infty}^{\overline{\eta}p}\big]\big)^{\eta/\overline{\eta}}
\end{align*}
with $c_{p,1}:=2^{2p-1}c_{0}^{p}(2^{p} + c^{2p})(1 + \underline{c}_{\overline{\eta} p})^{\eta/\overline{\eta}}$. We note that $|\null_{n,j}\Delta_{s,u}|\leq \lambda_{0}d_{\infty}((s,\null_{n}Y),(\underline{s}_{n},\null_{n}Y))$ for each $j\in\{1,\dots,k_{n}\}$ and all $s,u\in [r,t_{j,n})$ and by setting $\overline{c}_{p}:=2^{3p/2}\lambda_{0}^{p}(1 + \underline{c}_{\overline{\eta}p})^{1/\overline{\eta}}$, we obtain that
\begin{equation*}
\lambda_{0}^{p}\big(E\big[d_{\infty}((s,\null_{n}Y),(\underline{s}_{n},\null_{n}Y))^{2p}\big]\big)^{1/2}\leq \overline{c}_{p}|\mathbb{T}_{n}|^{p/2}\big(1 + E\big[\|\null_{n}Y^{r}\|_{\infty}^{\overline{\eta}p}\big]\big)^{1/\overline{\eta}}
\end{equation*}
for each $s\in [r,T]$. Consequently, the Cauchy-Schwarz inequality gives us the following bound for the third and sixth expression in the decomposition~\eqref{Ito Formula Decomposition}:
\begin{align*}
&E\bigg[\max_{j\in\{1,\dots,k_{n}\}}\int_{r}^{t_{j,n}}\bigg|\int_{\underline{s}_{n}}^{s}\null_{n,j}\Delta_{s,v}\bigg(\underline{B}(v,v,\null_{n}Y) + \int_{r}^{v}\partial_{v}\underline{B}(v,u,\null_{n}Y)\,du\bigg)\,dv\bigg|^{p}\,ds\bigg]\\
&\leq 2^{p-1}c^{p}\int_{r}^{T}(s-\underline{s}_{n})^{p}\lambda_{0}^{p}E\big[d_{\infty}((s,\null_{n}Y),(\underline{s}_{n},\null_{n}Y))^{p}(1 + \|\null_{n}Y^{s}\|_{\infty}^{\kappa})^{p}\big]\,ds\\
&\quad + 2^{p-1}c^{p}\int_{r}^{T}\lambda_{0}^{p}E\bigg[d_{\infty}((s,\null_{n}Y),(\underline{s}_{n},\null_{n}Y))^{p}(s-\underline{s}_{n})^{p-1}\int_{\underline{s}_{n}}^{s}\bigg(\int_{r}^{v}1 + \|\null_{n}Y^{u}\|_{\infty}^{\kappa}\,du\bigg)^{p}\,dv\bigg]\,ds\\
&\leq c_{p,2}|\mathbb{T}_{n}|^{p}\big(1 + E[\|\null_{n}Y^{r}\|_{\infty}^{\overline{\eta}p}\big)^{2/\overline{\eta}}\int_{r}^{T}(s-\underline{s}_{n})^{p/2}\,ds
\end{align*}
for $c_{p,2}:=2^{5p/2-1}(1 + (T-r)^{p})c^{p}(1 + \underline{c}_{\overline{\eta}p})^{1/\overline{\eta}}\overline{c}_{p}$. For the fourth expression we apply the Cauchy-Schwarz inequality twice, which entails that
\begin{align*}
&E\bigg[\max_{j\in\{1,\dots,k_{n}\}}\int_{r}^{t_{j,n}}\bigg|\int_{\underline{s}_{n}}^{s}\null_{n,j}\Delta_{s,u}B_{H}(u,u,\null_{n}Y)\,dh(u)\bigg|^{p}\,ds\bigg]\\
&\leq \|h\|_{1,2,r}^{p}c^{p}\int_{r}^{T}(s-\underline{s}_{n})^{p/2}\lambda_{0}^{p}E\big[d_{\infty}((s,\null_{n}Y),(\underline{s}_{n},\null_{n}Y))^{p}(1 + \|\null_{n}Y\|_{\infty}^{\kappa})^{p}\big]\,ds\\
&\leq c_{p,3}|\mathbb{T}_{n}|^{p}\big(1 + E\big[\|\null_{n}Y^{r}\|_{\infty}^{\overline{\eta}p}\big]\big)^{2/\overline{\eta}},
\end{align*}
where $c_{p,3}:=2^{3p/2}\|h\|_{1,2,r}^{p}c^{p}(1 + \underline{c}_{\overline{\eta}p})^{1/\overline{\eta}}\overline{c}_{p}$. Proceeding similarly, it follows for the seventh expression that
\begin{align*}
&E\bigg[\max_{j\in\{1,\dots,k_{n}\}}\int_{r}^{t_{j,n}}\bigg|\int_{\underline{s}_{n}}^{s}\null_{n,j}\Delta_{s,v}\int_{r}^{v}\partial_{v}B_{H}(v,u,\null_{n}Y)\,dh(u)\,dv\bigg|^{p}\,ds\bigg]\\
&\leq \int_{r}^{T}(s-\underline{s}_{n})^{p-1}\int_{\underline{s}_{n}}^{s}\lambda_{0}^{p}E\bigg[d_{\infty}((s,\null_{n}Y),(\underline{s}_{n},\null_{n}Y))^{p}\bigg|\int_{r}^{v}\partial_{v}B_{H}(v,u,\null_{n}Y)\,dh(u)\bigg|^{p}\bigg]\,dv\,ds\\
&\leq c_{p,4}|\mathbb{T}_{n}|^{p}\big(1 + E\big[\|\null_{n}Y^{r}\|_{\infty}^{\overline{\eta}p}\big]\big)^{2/\overline{\eta}}\int_{r}^{T}(s-\underline{s}_{n})^{p/2}\,ds
\end{align*}
with $c_{p,4}:=(T-r)^{p/2}c_{p,3}$. We turn to the fifth and eight term in~\eqref{Ito Formula Decomposition} and once again apply the Cauchy-Schwarz inequality, which leads us to
\begin{align*}
&E\bigg[\max_{j\in\{1,\dots,k_{n}\}}\int_{r}^{t_{j,n}}\bigg|\int_{\underline{s}_{n}}^{s}\null_{n,j}\Delta_{s,v}\bigg(\overline{B}(v,v,\null_{n}Y)\null_{n}\dot{W}_{v} + \int_{r}^{v}\partial_{v}\overline{B}(v,u,\null_{n}Y)\,d\null_{n}W_{u}\bigg)\,dv\bigg|^{p}\,ds\bigg]\\
&\leq 2^{p-1}c^{p}\int_{r}^{T}(s-\underline{s}_{n})^{p/2}\lambda_{0}^{p}E\bigg[d_{\infty}((s,\null_{n}Y),(\underline{s}_{n},\null_{n}Y))^{p}\bigg(\int_{\underline{s}_{n}}^{s}|\null_{n}\dot{W}_{v}|^{2}\,dv\bigg)^{p/2}\bigg]\,ds\\
&\quad + 2^{p-1}c^{p}\int_{r}^{T}(s-\underline{s}_{n})^{p-1}\int_{\underline{s}_{n}}^{s}\lambda_{0}^{p}E\bigg[d_{\infty}((s,\null_{n}Y),(\underline{s}_{n},\null_{n}Y))^{p}\bigg(\int_{r}^{v}|\null_{n}\dot{W}_{u}|\,du\bigg)^{p}\bigg]\,dv\,ds\\
&\leq c_{p,5}|\mathbb{T}_{n}|^{p}\big(1 + E\big[\|\null_{n}Y^{r}\|_{\infty}^{\overline{\eta}p}\big]\big)^{1/\overline{\eta}}
\end{align*}
for $c_{p,5}:=2^{2p-1}\hat{w}_{p,2}^{1/2}(1 + (T-r)^{p+1}/(p+1))c^{p}\overline{c}_{p}$. By using the constant $\overline{c}_{\mathbb{T}}$ appearing in condition~\eqref{C.4}, we derive the following estimate for the ninth term:
\begin{align*}
&E\bigg[\max_{j\in\{1,\dots,k_{n}\}}\int_{r}^{t_{j,n}}\bigg|\int_{\underline{s}_{n}}^{s}\null_{n,j}\Delta_{s,u}\Sigma(u,u,\null_{n}Y)\,dW_{u}\bigg|^{p}\,ds\bigg]\\
&\leq w_{p}c^{p}\sum_{j=1}^{k_{n}}\int_{r}^{t_{j,n}}(s-\underline{s}_{n})^{p/2-1}\int_{\underline{s}_{n}}^{s}\lambda_{0}^{p}E\big[d_{\infty}((s,\null_{n}Y),(\underline{s}_{n},\null_{n}Y))^{p}\big]\,du\,ds\\
&\leq c_{p,6}|\mathbb{T}_{n}|^{p-1}\big(1 + E\big[\|\null_{n}Y^{r}\|_{\infty}^{\overline{\eta}p}\big])^{1/\overline{\eta}},
\end{align*}
where $c_{p,6}:=2^{p/2}w_{p}c^{p}\underline{c}_{p}(T-r)\overline{c}_{\mathbb{T}}$. Finally, for the last expression we now readily estimate that
\begin{align*}
&E\bigg[\max_{j\in\{1,\dots,k_{n}\}}\int_{r}^{t_{j,n}}\bigg|\int_{\underline{s}_{n}}^{s}\null_{n,j}\Delta_{s,v}\int_{r}^{v}\partial_{v}\Sigma(v,u,\null_{n}Y)\,dW_{u}\,dv\bigg|^{p}\,ds\bigg]\\
&\leq \int_{r}^{T}(s-\underline{s}_{n})^{p-1}\int_{\underline{s}_{n}}^{s}E\bigg[\lambda_{0}^{p}d_{\infty}((s,\null_{n}Y),(\underline{s}_{n},\null_{n}Y))^{p}\bigg|\int_{r}^{v}\partial_{v}\Sigma(v,u,\null_{n}Y)\,dW_{u}\bigg|^{p}\bigg]\,dv\,ds\\
&\leq c_{p,7}|\mathbb{T}_{n}|^{p}\big(1 + E\big[\|\null_{n}Y^{r}\|_{\infty}^{\overline{\eta}p}\big]\big)^{1/\overline{\eta}}\int_{r}^{T}(s-\underline{s}_{n})^{p/2}\,ds
\end{align*}
for $c_{p,7}:=2^{p/2}c^{p}\overline{c}_{p}w_{2p}^{1/2}(T-r)^{p/2}$. So, we let $c_{p,8}:=(T-r)((T-r)c_{p,1} + c_{p,3} + c_{p,5})$ and $c_{p,9}:=(T-r)^{p/2 +2}(c_{p,2} + c_{p,4} + c_{p,7})$ and conclude by setting $c_{p}:=7^{p-1}(c_{p,6}+c_{p,8} + c_{p,9})$.
\end{proof}

Next, we give a bound for the second remainder in~\eqref{Main Remainder Decomposition}, which allows for another application of Lemma~\ref{Auxiliary Lemma 1}, according to Remark~\ref{Auxiliary Remark}.

\begin{Lemma}\label{Remainder Lemma 2}
Let~\eqref{C.5}-\eqref{C.7} be valid and $h\in W_{r}^{1,2}([0,T],\mathbb{R}^{d})$. Then for each $p\geq 2$ there is $c_{p} > 0$ such that each $n\in\mathbb{N}$ and any solution $\null_{n}Y$ to~\eqref{Sequence VIE} satisfy
\[
E[|\null_{n}Y_{s} - \null_{n}Y_{\underline{s}_{n}} - \Phi_{h,n}(s,\null_{n}Y,W)|^{p}] \leq c_{p}|\mathbb{T}_{n}|^{p}\big(1 + E\big[\|\null_{n}Y^{r}\|_{\infty}^{2p}\big]\big)^{1/2}
\]
for every $s\in [r,T)$.
\end{Lemma}

\begin{proof}
From Fubini's theorem for deterministic and stochastic integrals and the definition of $\Phi_{h,n}$ we get that
\begin{equation}\label{First-Order Remainder Decomposition}
\begin{split}
&\null_{n}Y_{s} - \null_{n}Y_{\underline{s}_{n}} - \Phi_{h,n}(s,\null_{n}Y,W) = \int_{\underline{s}_{n}}^{s}\underline{B}(u,u,\null_{n}Y)\,du\\
& + \int_{\underline{s}_{n}}^{s} B_{H}(u,u,\null_{n}Y) - B_{H}(\underline{s}_{n},\underline{s}_{n},\null_{n}Y)\,dh(u)\\
& + \int_{\underline{s}_{n}}^{s}\overline{B}(u,u,\null_{n}Y) - \overline{B}(\underline{s}_{n},\underline{s}_{n},\null_{n}Y)\,d\null_{n}W_{u}\\
& + \int_{\underline{s}_{n}}^{s}\int_{r}^{v}\partial_{v}\underline{B}(v,u,\null_{n}Y) + \partial_{v}B_{H}(v,u,\null_{n}Y)\,dh(u)\,dv\\
& + \int_{\underline{s}_{n}}^{s}\int_{\underline{s}_{n}}^{v}\partial_{v}\overline{B}(v,u,\null_{n}Y)\,d\null_{n}W_{u}\,dv\\
& + \int_{\underline{s}_{n}}^{s}\Sigma(u,u,\null_{n}Y) - \Sigma(\underline{s}_{n},\underline{s}_{n},\null_{n}Y)\,dW_{u} + \int_{\underline{s}_{n}}^{s}\int_{r}^{v}\partial_{v}\Sigma(v,u,\null_{n}Y)\,dW_{u}\,dv\quad\text{a.s.}
\end{split}
\end{equation}
Proposition~\ref{Sequence L-p Estimate Proposition} provides a constant $\underline{c}_{2p} > 0$ such that~\eqref{Sequence L-p Estimate} holds when $p$ and $c_{p}$ are replaced by $2p$ and $\underline{c}_{2p}$, respectively. We set $\overline{c}_{p,2}:=\lambda^{p} + (T-r)^{p/2}c^{p}$ and $\overline{c}_{p,1}:=(1 + \underline{c}_{2p})^{1/2}$ and define eight constants as follows:
\begin{align*}
&c_{p,1}:=2^{2p}c^{p}\overline{c}_{p,1},\,\,c_{p,2}:=2^{3p}\|h\|_{1,2,r}^{p}\overline{c}_{p,1}\overline{c}_{p,2},\,\,c_{p,3}:=2^{3p/2}3^{p}\hat{w}_{p,2}^{1/2}\overline{c}_{p,1}\overline{c}_{p,2},\\
&c_{p,4}:=(T-r)^{p}c_{p,1},\,\, c_{p,5}:=2^{2p}(T-r)^{p/2}\|h\|_{1,2,r}^{p}c^{p}\overline{c}_{p,1},\,\, c_{p,6}:=2^{3p/2}\hat{w}_{p,1}(T-r)^{p/2}c^{p},\\
&c_{p,7}:=2^{p}3^{p}w_{p}\overline{c}_{p,1}\overline{c}_{p,2}\,\,\text{and}\,\,c_{p,8}:=2^{p}w_{p}(T-r)^{p/2}c^{p}.
\end{align*}
By using the inequalities of Jensen and Cauchy-Schwarz and~\eqref{Auxiliary Estimate 3}, it follows readily that the $p$-th moment of the $i$-th expression in the decomposition~\eqref{First-Order Remainder Decomposition} is bounded by $c_{p,i}|\mathbb{T}_{n}|^{p}(1 + E[\|\null_{n}Y^{r}\|_{\infty}^{2p}])^{1/2}$ for all $i\in\{1,\dots,8\}$. We set $c_{p}:=8^{p-1}(c_{p,1} + \cdots + c_{p,8})$ and the asserted estimate follows.
\end{proof}

\subsection{A second moment estimate for the third remainder}\label{Section 4.3}

We directly bound the third remainder in~\eqref{Main Remainder Decomposition} by repeatedly using an estimate that follows for any $n\in\mathbb{N}$ with $k_{n}\geq 2$ from Doob's $L^{2}$-maximal inequality; see~\cite{ContKalininSupp}[Lemma 33] for details.

\begin{enumerate}[(i)]
\setcounter{enumi}{4}
\item For every $l\in\{1,\dots,d\}$ assume that $(\null_{l}U_{i})_{i\in\{1,\dots,k_{n}-1\}}$ and $(\null_{l}V_{i})_{i\in\{1,\dots,k_{n}-1\}}$ are two sequences of $\mathbb{R}^{1\times m}$-valued and $\mathbb{R}^{m}$-valued random vectors, respectively, such that $\null_{l}U_{i}$ is $\mathscr{F}_{t_{i-1,n}}$-measurable, $\null_{l}V_{i}$ is $\mathscr{F}_{t_{i},n}$-measurable,
\[
E[|\null_{l}U_{i}|^{4} + |\null_{l}U_{i}|^{4}] < \infty\quad\text{and}\quad E[\null_{l}V_{i}|\mathscr{F}_{t_{i-1,n}}] = 0\quad\text{a.s.}
\]
for all $i\in\{1,\dots,k_{n}-1\}$. Then
\begin{equation}\label{Auxiliary Estimate 5}
E\bigg[\max_{j\in\{1,\dots,k_{n}\}}\bigg|\sum_{i=1}^{j-1}\sum_{l=1}^{d}\null_{l}U_{i}\,\null_{l}V_{i}\bigg|^{2}\bigg] \leq 4 \sum_{i=1}^{k_{n}-1}\sum_{l_{1},l_{2}=1}^{d}E\big[\null_{l_{1}}U_{i}\,\null_{l_{1}}V_{i}\,\null_{l_{2}}V_{i}'\null_{l_{2}}U_{i}'\big].
\end{equation}
\end{enumerate}

\begin{Proposition}\label{Remainder Proposition 3}
Let~\eqref{C.5}-\eqref{C.8} be satisfied and $h\in W_{r}^{1,2}([0,T],\mathbb{R}^{d})$. Then there is $c_{2} > 0$ such that for each $n\in\mathbb{N}$ and any solution $\null_{n}Y$ to~\eqref{Sequence VIE} it holds that
\begin{align*}
E\bigg[\max_{j\in\{0\,\dots,k_{n}\}}\bigg|\int_{r}^{t_{j,n}}\partial_{x}\overline{B}(t_{j,n},\underline{s}_{n},\null_{n}Y)&\Phi_{h,n}(s,\null_{n}Y,W)\null_{n}\dot{W}_{s} - R(t_{j,n},\underline{s}_{n},\null_{n}Y)\gamma_{n}(s)\,ds\bigg|^{2}\bigg]\\
&\leq c_{2}|\mathbb{T}_{n}|(1 + E[\|\null_{n}Y^{r}\|_{\infty}^{2}]). 
\end{align*}
\end{Proposition}

\begin{proof}
By the definition~\eqref{General Remainder Definition} of the mapping $R$, we can write the $k$-th coordinate of $\partial_{x}\overline{B}(t_{j,n},\underline{s}_{n},\null_{n}Y)\Phi_{h,n}(s,\null_{n}Y,W)\null_{n}\dot{W}_{s} - R(t_{j,n},\underline{s}_{n},\null_{n}Y)\gamma_{n}(s)$ in the form
\[
\sum_{l=1}^{d}\partial_{x}\overline{B}_{k,l}(t_{j,n},\underline{s}_{n},\null_{n}Y)\big(\Phi_{h,n}(s,\null_{n}Y,W)\null_{n}\dot{W}_{s}^{(l)} - ((1/2)\overline{B} + \Sigma)(\underline{s}_{n},\underline{s}_{n},\null_{n}Y)\gamma_{n}(s)e_{l}\big)
\]
for each $j\in\{1,\dots,k_{n}\}$, any $k\in\{1,\dots,m\}$ and all $s\in [r,t_{j,n})$, where we write $X^{(l)}$ for the $l$-th coordinate of any $\mathbb{R}^{d}$-valued process $X$ for each $l\in\{1,\dots,d\}$. Based on this identity, we use the following decomposition:
\begin{equation}\label{Differential Remainder Decomposition}
\begin{split}
&\Phi_{h,n}(s,\null_{n}Y,W)\null_{n}\dot{W}_{s}^{(l)} - ((1/2)\overline{B} + \Sigma)(\underline{s}_{n},\underline{s}_{n},\null_{n}Y)\gamma_{n}(s)e_{l}\\
&=B_{H}(\underline{s}_{n},\underline{s}_{n},\null_{n}Y)(h(s_{n}) - h(\underline{s}_{n}))\null_{n}\dot{W}_{s}^{(l)} + \overline{B}(\underline{s}_{n},\underline{s}_{n},\null_{n}Y)(\null_{n}W_{s_{n}} - \null_{n}W_{\underline{s}_{n}})\null_{n}\dot{W}_{s}^{(l)}\\
&\quad + \Sigma(\underline{s}_{n},\underline{s}_{n},\null_{n}Y)\big(\Delta W_{s_{n}} \null_{n}\dot{W}_{s}^{(l)} - \gamma_{n}(s)e_{l}\big) + B_{H}(\underline{s}_{n},\underline{s}_{n},\null_{n}Y)(h(s) - h(s_{n}))\null_{n}\dot{W}_{s}^{(l)}\\
&\quad + \overline{B}(\underline{s}_{n},\underline{s}_{n},\null_{n}Y)\big((\null_{n}W_{s} - \null_{n}W_{s_{n}})\null_{n}\dot{W}_{s}^{(l)} - (1/2)\gamma_{n}(s)e_{l}\big)\\
&\quad + \Sigma(\underline{s}_{n},\underline{s}_{n},\null_{n}Y)(W_{s} - W_{s_{n}})\null_{n}\dot{W}_{s}^{(l)} + \bigg(\int_{\underline{s}_{n}}^{s}\int_{r}^{\underline{s}_{n}}\partial_{v}\overline{B}(v,u,\null_{n}Y)\,d\null_{n}W_{u}\,dv\bigg)\null_{n}\dot{W}_{s}^{(l)}
\end{split}
\end{equation}
with $l\in\{1,\dots,d\}$. To handle the first appearing term, we decompose the integral and apply Fubini's theorem for stochastic integrals to rewrite that
\begin{align*}
&\int_{r}^{t_{j,n}}\partial_{x}\overline{B}_{k,l}(t_{j,n},\underline{s}_{n},\null_{n}Y)B_{H}(\underline{s}_{n},\underline{s}_{n},\null_{n}Y)(h(s_{n}) - h(\underline{s}_{n}))\,d\null_{n}W_{s}^{(l)}\\
&=\int_{r}^{t_{j-1,n}}\partial_{x}\overline{B}_{k,l}(s,s_{n},\null_{n}Y)B_{H}(s_{n},s_{n},\null_{n}Y)(h(\overline{s}_{n}) - h(s_{n}))\,dW_{s}^{(l)}\\
&\quad + \int_{r}^{t_{j,n}}\int_{r}^{t\wedge t_{j-1,n}}\partial_{t}\partial_{x}\overline{B}_{k,l}(t,s_{n},\null_{n}Y)B_{H}(s_{n},s_{n},\null_{n}Y)(h(\overline{s}_{n}) - h(s_{n}))\,dW_{s}^{(l)}\,dt\quad\text{a.s.}
\end{align*}
for any $j\in\{1,\dots,k_{n}\}$, every $k\in\{1,\dots,m\}$ and each $l\in\{1,\dots,d\}$. By Proposition~\ref{Sequence L-p Estimate Proposition}, there is $\underline{c}_{2} > 0$ such that~\eqref{Sequence L-p Estimate} holds for $p=2$ with $\underline{c}_{2}$ instead of $c_{p}$. Therefore,
\begin{align*}
&E\bigg[\max_{j\in\{0,\dots,k_{n}\}}\sum_{k=1}^{m}\bigg|\sum_{l=1}^{d}\int_{r}^{t_{j,n}}\partial_{x}\overline{B}_{k,l}(t_{j,n},\underline{s}_{n},\null_{n}Y)B_{H}(\underline{s}_{n},\underline{s}_{n},\null_{n}Y)(h(s_{n}) - h(\underline{s}_{n}))\,d\null_{n}W_{s}^{(l)}\bigg|^{2}\bigg]\\
&\leq 2w_{2}c^{2}\overline{c}^{2}\int_{r}^{T}E\big[\big(1 + \|\null_{n}Y^{s_{n}}\|_{\infty}^{\kappa}\big)^{2}\big] |h(\overline{s}_{n}) - h(s_{n})|^{2}\,ds\\
&\quad + 2w_{2}(T-r)c^{2}\overline{c}^{2}\int_{r}^{T}\int_{r}^{t}E\big[\big(1 + \|\null_{n}Y^{s_{n}}\|_{\infty}^{\kappa}\big)^{2}\big]|h(\overline{s}_{n})-h(s_{n})|^{2}\,ds\,dt\\
&\leq c_{2,1}|\mathbb{T}_{n}|\big(1 + E\big[\|\null_{n}Y^{r}\|_{\infty}^{2}\big]\big)^{\kappa}
\end{align*}
with $c_{2,1}:= 2^{3}w_{2}(1 + (T - r)^{2}/2)(T-r)\|h\|_{1,2,r}^{2}c^{2}\overline{c}^{2}(1 + \underline{c}_{2})^{\kappa}$. Proceeding similarly, we obtain for the second term in the decomposition~\eqref{Differential Remainder Decomposition} that
\begin{align*}
&\int_{r}^{t_{j,n}}\partial_{x}\overline{B}_{k,l}(t_{j,n},\underline{s}_{n},\null_{n}Y)\overline{B}(\underline{s}_{n},\underline{s}_{n},\null_{n}Y)(\null_{n}W_{s_{n}} - \null_{n}W_{\underline{s}_{n}})\,d\null_{n}W_{s}^{(l)}\\
&= \int_{r}^{t_{j-1,n}}\partial_{x}\overline{B}_{k,l}(s,s_{n},\null_{n}Y)\overline{B}(s_{n},s_{n},\null_{n}Y)\Delta W_{s_{n}}\,dW_{s}^{(l)}\\
&\quad + \int_{r}^{t_{j,n}}\int_{r}^{t\wedge t_{j-1,n}}\partial_{t}\partial_{x}\overline{B}_{k,l}(t,s_{n},\null_{n}Y)\overline{B}(s_{n},s_{n},\null_{n}Y)\Delta W_{s_{n}}\,dW_{s}^{(l)}\,dt\quad\text{a.s.}
\end{align*}
for every $j\in\{1,\dots,k_{n}\}$, each $k\in\{1,\dots,m\}$ and any $l\in\{1,\dots,d\}$. Hence, by setting $c_{2,2}:=2w_{2}(1 + (T - r)^{2}/2)(T-r)d c^{2}\overline{c}^{2}$, it follows readily that
\begin{align*}
&E\bigg[\max_{j\in\{0,\dots,k_{n}\}}\sum_{k=1}^{m}\bigg|\sum_{l=1}^{d}\int_{r}^{t_{j,n}}\partial_{x}\overline{B}_{k,l}(t_{j,n},\underline{s}_{n},\null_{n}Y)\overline{B}(\underline{s}_{n},\underline{s}_{n},\null_{n}Y)(\null_{n}W_{s_{n}} - \null_{n}W_{\underline{s}_{n}})\,d\null_{n}W_{s}^{(l)}\bigg|^{2}\bigg]\\
&\leq 2w_{2}c^{2}\overline{c}^{2}\int_{r}^{T}\bigg(E\big[|\Delta W_{t_{n}}|^{2}\big]+ (T-r)\int_{r}^{t}E\big[|\Delta W_{s_{n}}|^{2}\big]\,ds\bigg)\,dt \leq c_{2,2}|\mathbb{T}_{n}|.
\end{align*} 

To deal with the third term in~\eqref{Differential Remainder Decomposition}, we utilize the $\mathbb{R}^{d}$-valued $\mathscr{F}_{t_{i,n}}$-measurable random vector
\[
\null_{l,n}V_{i}:=\Delta W_{t_{i,n}}\Delta W_{t_{i,n}}^{(l)} - \Delta t_{i,n}e_{l},
\]
which is independent of $\mathscr{F}_{t_{i-1,n}}$ and satisfies $E[\null_{l,n}V_{i}] = 0$ for any $i\in\{1,\dots,k_{n}\}$ and each $l\in\{1,\dots,d\}$. We note that if $\mathbb{I}_{l_{2},l_{1}}\in\mathbb{R}^{d\times d}$ denotes the matrix whose $(l_{2},l_{1})$-entry is $1$ and whose all other entries are zero, then
\[
E[_{l_{1},n}V_{i}\,\null_{l_{2},n}V_{i}'] = \mathbbm{1}_{\{l_{2}\}}(l_{1})(\Delta t_{i,n})^{2}(\mathbbm{I}_{d} + \mathbbm{I}_{l_{2},l_{1}})
\]
whenever $i\in\{1,\dots,k_{n}\}$ and $l_{1},l_{2}\in\{1,\dots,d\}$. Now, by decomposing the integral once again, we obtain that
\begin{align*}
&\int_{r}^{t_{j,n}}\partial_{x}\overline{B}_{k,l}(t_{j,n},\underline{s}_{n},\null_{n}Y)\Sigma(\underline{s}_{n},\underline{s}_{n},\null_{n}Y)\big(\Delta W_{s_{n}}\null_{n}\dot{W}_{s}^{(l)} - \gamma_{n}(s)e_{l}\big)\,ds\\
&= \sum_{i=1}^{j-1}\partial_{x}\overline{B}_{k,l}(t_{i,n},t_{i-1,n},\null_{n}Y)\Sigma(t_{i-1,n},t_{i-1,n},\null_{n}Y)\null_{l,n}V_{i}\\
&\quad + \sum_{i_{2}=1}^{j-1}\int_{t_{i_{2},n}}^{t_{i_{2}+1,n}}\sum_{i_{1}=1}^{i_{2}}\partial_{t}\partial_{x}\overline{B}_{k,l}(t,t_{i_{1}-1,n},\null_{n}Y)\Sigma(t_{i_{1}-1,n},t_{i_{1}-1,n},\null_{n}Y)\,\null_{l,n}V_{i_{1}}\,dt
\end{align*}
for all $j\in\{1,\dots,k_{n}\}$, each $k\in\{1,\dots,m\}$ and every $l\in\{1,\dots,d\}$. Consequently, the estimate~\eqref{Auxiliary Estimate 5} and Young's inequality give us that
\begin{align*}
&E\bigg[\max_{j\in\{0,\dots,k_{n}\}}\sum_{k=1}^{m}\bigg|\int_{r}^{t_{j,n}}\sum_{l=1}^{d}\partial_{x}\overline{B}_{k,l}(t_{j,n},\underline{s}_{n},\null_{n}Y)\Sigma(\underline{s}_{n},\underline{s}_{n},\null_{n}Y)\big(\Delta W_{s_{n}}\null_{n}\dot{W}_{s}^{(l)} - \gamma_{n}(s)e_{l}\big)\,ds\bigg|^{2}\bigg]\\
&\leq 2^{4}|\mathbb{T}_{n}|\sum_{i=1}^{k_{n}-1}\Delta t_{i,n}\sum_{k=1}^{m}\sum_{l=1}^{d} E\big[|\partial_{x}\overline{B}_{k,l}(t_{i,n},t_{i-1,n},\null_{n}Y)\Sigma(t_{i-1,n},t_{i-1,n},\null_{n}Y)\big|^{2}\big]\\
&\quad + 2^{4}(T-r)\int_{r}^{T}\sum_{i=1}^{k_{n}-1}(\Delta t_{i,n})^{2}\sum_{k=1}^{m}\sum_{l=1}^{d}E\big[\big|\partial_{t}\partial_{x}\overline{B}_{k,l}(t,t_{i-1,n},\null_{n}Y)\Sigma(t_{i-1,n},t_{i-1,n},\null_{n}Y)\big|^{2}\big]\,dt\\
&\leq c_{2,3}|\mathbb{T}_{n}|,
\end{align*}
where $c_{2,3}:=2^{4}(1 + (T-r)^{2})(T-r)c^{2}\overline{c}^{2}$. For the fourth expression in~\eqref{Differential Remainder Decomposition} we integrate by parts, after another decomposition of the integral, which yields that
\begin{align*}
&\int_{r}^{t_{j,n}}\partial_{x}\overline{B}(t_{j,n},\underline{s}_{n},\null_{n}Y)B_{H}(\underline{s}_{n},\underline{s}_{n},\null_{n}Y)(h(s) - h(s_{n}))\,d\null_{n}W_{s}^{(l)}\\
&= \int_{r}^{t_{j,n}}\partial_{x}\overline{B}_{k,l}(s,\underline{s}_{n},\null_{n}Y)B_{H}(\underline{s}_{n},\underline{s}_{n},\null_{n}Y)\Delta W_{s_{n}}^{(l)}\frac{(\overline{s}_{n} -s)}{\Delta\overline{s}_{n}}\,dh(s)\\
&\quad + \int_{r}^{t_{j,n}}\int_{r}^{t}\partial_{t}\partial_{x}\overline{B}_{k,l}(t,\underline{s}_{n},\null_{n}Y)B_{H}(\underline{s}_{n},\underline{s}_{n},\null_{n}Y)\Delta W_{s_{n}}^{(l)}\frac{(\overline{s}_{n} -s)}{\Delta\overline{s}_{n}}\,dh(s)\,dt
\end{align*}
for each $j\in\{1,\dots,k_{n}\}$, any $k\in\{1,\dots,m\}$ and every $l\in\{1,\dots,d\}$. Hence, from the Cauchy-Schwarz inequality we get that
\begin{align*}
&E\bigg[\max_{j\in\{0,\dots,k_{n}\}}\sum_{k=1}^{m}\bigg|\sum_{l=1}^{d}\int_{r}^{t_{j,n}}\partial_{x}\overline{B}(t_{j,n},\underline{s}_{n},\null_{n}Y)B_{H}(\underline{s}_{n},\underline{s}_{n},\null_{n}Y)(h(s) - h(s_{n}))\,d\null_{n}W_{s}^{(l)}\bigg|^{2}\bigg]\\
&\leq 2\|h\|_{1,2,r}^{2} \int_{r}^{T}\sum_{k=1}^{m}\sum_{l=1}^{d}E\big[\big|\partial_{x}\overline{B}_{k,l}(s,\underline{s}_{n},\null_{n}Y)B_{H}(\underline{s}_{n},\underline{s}_{n},\null_{n}Y)|^{2}\big] E\big[|\Delta W_{s_{n}}^{(l)}\big|^{2}\big]\,ds\\
&\quad +2\|h\|_{1,2,r}^{2}(T-r)\int_{r}^{T}\int_{r}^{t}\sum_{k=1}^{m}E\bigg[\bigg|\sum_{l=1}^{d}\partial_{x}\overline{B}_{k,l}(t,\underline{s}_{n},\null_{n}Y)B_{H}(\underline{s}_{n},\underline{s}_{n},\null_{n}Y)\Delta W_{s_{n}}^{(l)}\bigg|^{2}\bigg]\,ds\,dt\\
&\leq  c_{2,4}|\mathbb{T}_{n}|\big(1 + E[\|\null_{n}Y^{r}\|_{\infty}^{2}]\big)^{\kappa}
\end{align*}
with $c_{2,4}:=2^{3}(1+(T-r)^{2}/2)(T-r)\|h\|_{1,2,r}^{2}c^{2}\overline{c}^{2}(1+\underline{c}_{2})^{\kappa}$, because $\Delta W_{n}^{(1)},\dots,\Delta W_{s_{n}}^{(d)}$ are pairwise independent and independent of $\mathscr{F}_{\underline{s}_{n}}$ for all $s\in [r,T]$. 

The fifth term in~\eqref{Differential Remainder Decomposition} can be treated in a similar way as the third. Namely, we set $\null_{l,n}U_{s}:=(\null_{n}W_{s} - \null_{n}W_{s_{n}})\null_{n}\dot{W}_{s}^{(l)} - (1/2)\gamma_{n}(s)e_{l}$ for every $s\in [r,T]$ and rewrite that
\begin{align*}
&\int_{r}^{t_{j,n}}\partial_{x}\overline{B}_{k,l}(t_{j,n},\underline{s}_{n},\null_{n}Y)\overline{B}(\underline{s}_{n},\underline{s}_{n},\null_{n}Y)\null_{l,n}U_{s}\,ds\\
&= \frac{1}{2}\sum_{i=1}^{j-1}\partial_{x}\overline{B}_{k,l}(t_{i,n},t_{i-1,n},\null_{n}Y)\overline{B}(t_{i-1,n},t_{i-1,n},\null_{n}Y)\null_{l,n}V_{i}\\
&\quad + \frac{1}{2}\sum_{i_{2}=1}^{j-1}\int_{t_{i_{2},n}}^{t_{i_{2}+1,n}}\sum_{i_{1}=1}^{i_{2}}\partial_{t}\partial_{x}\overline{B}_{k,l}(t,t_{i_{1}-1,n},\null_{n}Y)\overline{B}(t_{i_{1}-1,n},t_{i_{1}-1,n},\null_{n}Y)\null_{l,n}V_{i_{1}}\,dt
\end{align*}
for all $j\in\{1,\dots,k_{n}\}$, each $k\in\{1,\dots,m\}$ and every $l\in\{1,\dots,d\}$. Thus, from the estimate~\eqref{Auxiliary Estimate 5} we can again infer that
\begin{align*}
& E\bigg[\max_{j\in\{0,\dots,k_{n}\}}\sum_{k=1}^{m}\bigg|\int_{r}^{t_{j,n}}\sum_{l=1}^{d}\partial_{x}\overline{B}_{k,l}(t_{j,n},\underline{s}_{n},\null_{n}Y)\overline{B}(\underline{s}_{n},\underline{s}_{n},\null_{n}Y)\null_{l,n}U_{s}\,ds\bigg|^{2}\bigg]\\
&\leq 2^{2}|\mathbb{T}_{n}|\sum_{i=1}^{k_{n}-1}\Delta t_{i,n}\sum_{k=1}^{m}\sum_{l=1}^{d}E\big[\big|\partial_{x}\overline{B}_{k,l}(t_{i,n},t_{i-1,n},\null_{n}Y)\overline{B}(t_{i-1,n},t_{i-1,n},\null_{n}Y)\big|^{2}\big]\\
&\quad + 2^{2}(T-r)\int_{r}^{T}\sum_{i=1}^{k_{n}-1}(\Delta t_{i,n})^{2}\sum_{k=1}^{m}\sum_{l=1}^{d} E\big[|\partial_{t}\partial_{x}\overline{B}_{k,l}(t,t_{i-1,n},\null_{n}Y)\overline{B}(t_{i-1,n},t_{i-1,n},\null_{n}Y)\big|^{2}\big]\,dt\\
&\leq c_{2,5}|\mathbb{T}_{n}|
\end{align*}
for $c_{2,5}:=2^{2}(1 + (T-r)^{2})(T-r)c^{2}\overline{c}^{2}$. For the sixth expression in~\eqref{Differential Remainder Decomposition} we decompose the integral and apply It{\^o}'s formula to the effect that
\begin{align*}
&\int_{r}^{t_{j,n}}\partial_{x}\overline{B}_{k,l}(t_{j,n},\underline{s}_{n},\null_{n}Y)\Sigma(\underline{s}_{n},\underline{s}_{n},\null_{n}Y)(W_{s}-W_{s_{n}})\,d\null_{n}W_{s}^{(l)}\\
&= \int_{r}^{t_{j,n}}\partial_{x}\overline{B}_{k,l}(s,\underline{s}_{n},\null_{n}Y)\Sigma(\underline{s}_{n},\underline{s}_{n},\null_{n}Y)\frac{(\overline{s}_{n} - s)}{\Delta\overline{s}_{n}}\Delta W_{s_{n}}^{(l)}\,dW_{s}\\
&\quad + \int_{r}^{t_{j,n}}\int_{r}^{t}\partial_{t}\partial_{x}\overline{B}_{k,l}(t,\underline{s}_{n},\null_{n}Y)\Sigma(\underline{s}_{n},\underline{s}_{n},\null_{n}Y)\frac{(\overline{s}_{n} - s)}{\Delta\overline{s}_{n}}\Delta W_{s_{n}}^{(l)}\,dW_{s}\,dt\quad\text{a.s.}
\end{align*}
for all $j\in\{1,\dots,k_{n}\}$, each $k\in\{1,\dots,m\}$ and every $l\in\{1,\dots,d\}$. Hence, by utilizing that $\Delta W_{s_{n}}^{(1)},\dots,\Delta W_{s_{n}}^{(d)}$ are pairwise independent for any $s\in [r,T]$, we estimate that
\begin{align*}
& E\bigg[\max_{j\in\{0,\dots,k_{n}\}}\sum_{k=1}^{m}\bigg|\sum_{l=1}^{d}\int_{r}^{t_{j,n}}\partial_{x}\overline{B}_{k,l}(t_{j,n},\underline{s}_{n},\null_{n}Y)\Sigma(\underline{s}_{n},\underline{s}_{n},\null_{n}Y)(W_{s}-W_{s_{n}})\,d\null_{n}W_{s}^{(l)}\bigg|^{2}\bigg]\\
&\leq 2w_{2}|\mathbb{T}_{n}|\int_{r}^{T}\sum_{k=1}^{m}\sum_{l=1}^{d}E\big[\big|\partial_{x}\overline{B}_{k,l}(s,\underline{s}_{n},\null_{n}Y)\Sigma(\underline{s}_{n},\underline{s}_{n},\null_{n}Y)\big|^{2}\big]\,ds\\
&\quad + 2w_{2}(T-r)\int_{r}^{T}\int_{r}^{t}\sum_{k=1}^{m}\sum_{l=1}^{d}E\big[\big|\partial_{t}\partial_{x}\overline{B}_{k,l}(t,\underline{s}_{n},\null_{n}Y)\Sigma(\underline{s}_{n},\underline{s}_{n},\null_{n}Y)\big|^{2}\big]\Delta s_{n}\,ds\,dt\\
&\leq c_{2,6}|\mathbb{T}_{n}|,
\end{align*}
where $c_{2,6}:=2w_{2}(1 + (T-r)^{2}/2)(T-r)c^{2}\overline{c}^{2}$. 

Finally, for the seventh expression in~\eqref{Differential Remainder Decomposition} we define an $\mathbb{R}^{m}$-valued $\mathscr{F}_{t_{i-1,n}}$-measurable random vector by
\[
\null_{l,n}X_{i}:=\frac{1}{\Delta t_{i+1,n}}\int_{t_{i,n}}^{t_{i+1,n}}\int_{t_{i-1,n}}^{s}\int_{r}^{t_{i-1,n}}\partial_{v}\overline{B}(v,u,\null_{n}Y)\,d\null_{n}W_{u}\,dv\,ds,
\]
which satisfies $E\big[|\null_{l,n}X_{i}|^{2}\big] \leq 2^{2}\hat{w}_{2,1}(T-r)^{2}c_{\mathbb{T}}c^{2}(t_{i+1,n}-t_{i,n})$, for any $i\in\{1,\dots,k_{n}-1\}$ and every $l\in\{1,\dots,d\}$. Then we have
\begin{align*}
&\int_{r}^{t_{j,n}}\partial_{x}\overline{B}_{k,l}(t_{j,n},\underline{s}_{n},\null_{n}Y)\bigg(\int_{\underline{s}_{n}}^{s}\int_{r}^{\underline{s}_{n}}\partial_{v}\overline{B}(v,u,\null_{n}Y)\,d\null_{n}W_{u}\,dv\bigg)\,d\null_{n}W_{s}^{(l)}\\
&= \sum_{i=1}^{j-1}\partial_{x}\overline{B}_{k,l}(t_{i,n},t_{i-1,n},\null_{n}Y)\null_{l,n}X_{i}\Delta W_{t_{i,n}}^{(l)}\\
&\quad + \sum_{i_{2}=1}^{j-1}\int_{t_{i_{2},n}}^{t_{i_{2}+1,n}}\sum_{i_{1}=1}^{i_{2}}\partial_{t}\partial_{x}\overline{B}_{k,l}(t,t_{i_{1}-1,n},\null_{n}Y)\null_{l,n}X_{i_{1}}\Delta W_{t_{i_{1},n}}^{(l)}\,dt
\end{align*}
for all $j\in\{1,\dots,k_{n}\}$, each $k\in\{1,\dots,m\}$ and every $l\in\{1,\dots,d\}$. As $\Delta W_{t_{i,n}}^{(1)},\dots, W_{t_{i,n}}^{(d)}$ are pairwise independent and independent of $\mathscr{F}_{t_{i-1,n}}$ for every $i\in\{1,\dots,k_{n}\}$, it follows that
\begin{align*}
&E\bigg[\max_{j\in\{0,\dots,k_{n}\}}\sum_{k=1}^{m}\bigg|\sum_{l=1}^{d}\int_{r}^{t_{j,n}}\partial_{x}\overline{B}_{k,l}(t_{j,n},\underline{s}_{n},\null_{n}Y)\bigg(\int_{\underline{s}_{n}}^{s}\int_{r}^{\underline{s}_{n}}\partial_{v}\overline{B}(v,u,\null_{n}Y)\,d\null_{n}W_{u}\,dv\bigg)\,d\null_{n}W_{s}^{(l)}\bigg|^{2}\bigg]\\
&\leq 2^{3}\sum_{i=1}^{k_{n-1}}\Delta t_{i,n}\sum_{k=1}^{m}\sum_{l=1}^{d}E\big[|\partial_{x}\overline{B}_{k,l}(t_{i,n},t_{i-1,n},\null_{n}Y)\null_{l,n}X_{i}|^{2}\big]\\
&\quad + 2(T-r)\int_{r}^{T}\sum_{k=1}^{m}E\bigg[\max_{j\in\{1,\dots,k_{n}\}}\bigg|\sum_{i=1}^{j-1}\sum_{l=1}^{d}\partial_{t}\partial_{x}\overline{B}_{k,l}(t,t_{i-1,n},\null_{n}Y)\null_{l,n}X_{i}\Delta W_{t_{i,n}}^{(l)}\bigg|^{2}\bigg]\,dt\\
&\leq c_{2,7}|\mathbb{T}_{n}|
\end{align*}
with $c_{2,7}:=2^{5}(1 + (T-r)^{2})(T-r)^{3}\hat{w}_{2,1}c_{\mathbb{T}}c^{2}\overline{c}^{2}$, by virtue of the estimate~\eqref{Auxiliary Estimate 5}. Hence, we complete the proof by setting $c_{2}:=7(c_{2,1} + \cdots + c_{2,7})$.
\end{proof}

\section{Proofs of the convergence result in second moment and the support representation}\label{Section 5}

\subsection{Proofs of Lemmas~\ref{Auxiliary Convergence Lemma} and~\ref{Support Auxiliary Lemma}}\label{Section 5.1}

\begin{proof}[Proof of Lemma~\ref{Auxiliary Convergence Lemma}]
(i) If $\overline{b}_{0} = 0$ holds in~\eqref{C.9}, then~\eqref{Sequence VIE} reduces to a pathwise Volterra integral equation. In this case, pathwise uniqueness and strong existence are covered by the deterministic results in~\cite{KalininVolterra} or can essentially be inferred from~\cite{ProtterVolterra}. Otherwise, we may assume that $\overline{b}_{0}=1$ and introduce a martingale $\null_{n}\overline{Z}\in\mathscr{C}([0,T],\mathbb{R})$ by $\null_{n}\overline{Z}^{r} = 1$ and
\[
\null_{n}\overline{Z}_{t} = \exp\bigg(-\int_{r}^{t}\overline{b}(s)\null_{n}\dot{W}_{s}'\,dW_{s} - \frac{1}{2}\int_{r}^{t}|\overline{b}(s)\null_{n}\dot{W}_{s}|^{2}\,ds\bigg)
\]
for all $t\in [r,T]$ a.s. Then $\null_{n}\overline{W}\in\mathscr{C}([0,T],\mathbb{R}^{d})$ defined via $\null_{n}\overline{W}_{t} := W_{t} + \int_{r}^{r\vee t}\overline{b}(s)\,d\null_{n}W_{s}$ is a $d$-dimensional $(\mathscr{F}_{t})_{t\in [0,T]}$-Brownian motion under the probability measure $\overline{P}_{n}$ on $(\Omega,\mathscr{F})$ given by $\overline{P}_{n}(A):=E[\mathbbm{1}_{A}\,\null_{n}\overline{Z}_{T}]$, by Girsanov's theorem.

We observe that a process $Y\in\mathscr{C}([0,T],\mathbb{R}^{m})$ is a solution to~\eqref{Sequence VIE} under $P$ if and only if it solves the path-dependent stochastic Volterra integral equation
\begin{equation}
Y_{t} = Y_{r} + \int_{r}^{t}\underline{B}(t,s,Y) + B_{H}(t,s,Y)\dot{h}(s)\,ds + \int_{r}^{t}\Sigma(t,s,Y)\,d\null_{n}\overline{W}_{s}
\end{equation}
a.s.~for all $t\in [r,T]$ under $\overline{P}_{n}$. Consequently, pathwise uniqueness and strong existence follow from the stochastic results in~\cite{KalininVolterra} or~\cite{ProtterVolterra} when considering the drift $\underline{B} + B_{H}\dot{h}$ and the diffusion $\Sigma$.

Regarding the claimed estimate, we let $p > 2$ and $\alpha\in [0,1/2 - 1/p)$. Then from Proposition~\ref{Sequence L-p Estimate Proposition} we obtain $c_{p} > 0$ such that~\eqref{Sequence L-p Estimate} holds and the Kolmogorov-Chentsov estimate~\eqref{Kolmogorov-Chentsov Estimate} implies that
\[
E[(\|\null_{n}Y\|_{\alpha,r} - \|\null_{n}\xi^{r}\|_{\infty})^{p}] \leq k_{\alpha,p,p/2-1}c_{p}(T-r)^{p(1/2 - \alpha)}\big(1 + E[\|\null_{n}\xi^{r}\|_{\infty}^{p}]\big)
\]
for every $n\in\mathbb{N}$. Hence, we set $c_{\alpha,p}:=2^{p}k_{\alpha,p,p/2-1}(1 + c_{p})\max\{1,T-r\}^{p(1/2-\alpha)}$, then the triangle inequality gives the desired result.

(ii) Pathwise uniqueness, strong existence and the asserted bound can be directly inferred from (i) by replacing $\underline{B}$ by $\underline{B} + R$, $\overline{B}$ by $0$ and $\Sigma$ by $\overline{B} + \Sigma$, since~\eqref{C.9} holds in this case with $\overline{b} = 0$.
\end{proof}

\begin{proof}[Proof of Lemma~\ref{Support Auxiliary Lemma}]
(i) Pathwise uniqueness, the existence of a unique strong solution and the integrability condition follow from assertion (ii) of Lemma~\ref{Auxiliary Convergence Lemma} by letting $\underline{B} = b$, $B_{H} = \overline{B} = 0$, $\Sigma = \sigma$ and $\xi = \hat{x}$.

(ii) For $h\in W_{r}^{1,p}([0,T],\mathbb{R}^{d})$ we set $F_{h}:=b - (1/2)\rho + \sigma\dot{h}$. First, since $\partial_{x}\sigma(\cdot,s,x)$ is absolutely continuous on $[s,T)$, so is $\rho$ and hence, $F_{h}$ for any $s\in [r,T)$ and each $x\in C([0,T],\mathbb{R}^{m})$. Secondly, there are $c_{0},\lambda_{0}\geq 0$ such that $\max\{|\sigma|,|\partial_{t}\sigma|,|\rho|,|\partial_{t}\rho|\}\leq c_{0}$ and 
\[
|\rho(s,s,x) - \rho(s,s,y)| + |\partial_{t}\rho(t,s,x) - \partial_{t}\rho(t,s,y)| \leq \lambda_{0}\|x-y\|_{\infty}
\]
for all $s,t\in [r,T)$ with $s < t$ and every $x,y\in C([0,T],\mathbb{R}^{m})$. These conditions ensure that the map $F_{h}$ satisfies $|F_{h}(s,s,x)| + |\partial_{t}F_{h}(t,s,x)|\leq c_{1}(1 + |\dot{h}(s)|)\big(1 + \|x\|_{\infty}\big)$ and
\begin{equation*}
|F_{h}(s,s,x) - F_{h}(s,s,y)| + |\partial_{t}F_{h}(t,s,x) - \partial_{t}F_{h}(t,s,y)|\leq \lambda_{1}(1 + |\dot{h}(s)|)\|x-y\|_{\infty}
\end{equation*}
for any $s,t\in [r,T)$ with $s < t$ and each $x,y\in C([0,T],\mathbb{R}^{m})$, where $c_{1}:=3\max\{c_{0},c\}$ and $\lambda_{1}:=2\max\{\lambda_{0},\lambda\}$. As these are all the necessary assumptions, we invoke~\cite{KalininVolterra} to get a unique mild solution $x_{h}$ to~\eqref{Support VIE}, which satisfies $x_{h}\in W_{r}^{1,p}([0,T],\mathbb{R}^{m})$.

To show the the second assertion, we also let $g\in W_{r}^{1,p}([0,T],\mathbb{R}^{d})$. Then for the constant $c_{p,1}:=2^{2p-2}(1+T-r)^{p}\max\{1,T-r\}^{p-1}\max\{c_{0}^{p},\lambda_{1}^{p}\}$ we have
\[
\|x_{g}^{t} - x_{h}^{t}\|_{1,p,r}^{p} \leq c_{p,1} \int_{r}^{t}|\dot{g}(s) - \dot{h}(s)|^{p} + (1 + |\dot{h}(s)|^{p})\|x_{g}^{s}-x_{h}^{s}\|_{1,p,r}^{p}\,ds
\]
for each $t\in [r,T]$, since $\|y\|_{\infty}\leq \max\{1,T-r\}^{1-1/p}\|y\|_{1,p,r}$ for any $y\in W_{r}^{1,p}([0,T],\mathbb{R}^{m})$. Hence, Gronwall's inequality gives $\|x_{g} - x_{h}\|_{1,p,r}^{p}\leq c_{p}\exp(c_{p}\|h\|_{1,p,r}^{p})\|g-h\|_{1,p,r}^{p}$, where we have defined $c_{p}:=c_{p,1}\exp((T-r)c_{p,1})$.
\end{proof}

\subsection{Proofs of Theorems~\ref{Main Convergence Theorem} and~\ref{Support Theorem}}\label{Section 5.2}

\begin{proof}[Proof of Theorem~\ref{Main Convergence Theorem}]
By Lemma~\ref{Convergence along Partitions Lemma}, which is applicable due to Proposition~\ref{Sequence L-p Estimate Proposition} and Corollary~\ref{Sequence L-p Estimate Corollary}, we merely have to show the first assertion, as the second follows from the first. 

In this regard, the decomposition of Proposition~\ref{Main Decomposition Proposition} in second moment, Lemma~\ref{Auxiliary Lemma 2} and a combination of the estimate~\eqref{Auxiliary Estimate 1}  and Proposition~\ref{Auxiliary Proposition} with H\oe lder's inequality show that this limit holds once we can justify that there is $c_{2} > 0$ such that
\[
E\bigg[\max_{j\in\{0,\dots,k_{n}\}}\bigg|\int_{r}^{t_{j,n}}\big(\overline{B}(t_{j,n},s,\null_{n}Y)-\overline{B}(t_{j,n},\underline{s}_{n},\null_{n}Y)\big)\null_{n}\dot{W}_{s} - R(t_{j,n},\underline{s}_{n},\null_{n}Y)\gamma_{n}(s)\,ds\bigg|^{2}\bigg]
\]
does not exceed $c_{2}|\mathbb{T}_{n}|$ for every $n\in\mathbb{N}$. Based on the decomposition~\eqref{Main Remainder Decomposition} and the hypothesis that $\partial_{x}\overline{B}$ is bounded, this fact follows from Proposition~\ref{Remainder Proposition 1} and Lemma~\ref{Remainder Lemma 2}, in conjunction with Lemma~\ref{Auxiliary Lemma 1} and Remark~\ref{Auxiliary Remark}, and Proposition~\ref{Remainder Proposition 3}.
\end{proof}

\begin{proof}[Proof of Theorem~\ref{Support Theorem}]
We let $N_{\alpha}$ denote the $P$-null set of all $\omega\in\Omega$ such that $X(\omega)$ fails to be $\alpha$-H\oe lder continuous on $[r,T]$ and recall that the support of $P\circ X^{-1}$ in $C_{r}^{\alpha}([0,T],\mathbb{R}^{m})$ coincides with the support of the inner regular probability measure
\begin{equation}\label{Support Probability Measure}
\mathscr{B}(C_{r}^{\alpha}([0,T],\mathbb{R}^{m}))\rightarrow [0,1], \quad B\mapsto P(\{X\in B\}\cap N_{\alpha}^{c}).
\end{equation}
Then an application of Theorem~\ref{Main Convergence Theorem} in the case that $\underline{B}=b-(1/2)\rho$, $B_{H} = 0$, $\overline{B} = \sigma$, $\Sigma = 0$ and $\xi = \hat{x}$ gives us~\eqref{Hoelder Limit 1}, which in turn implies that the support of~\eqref{Support Probability Measure} is included in the closure of $\{x_{h}\,|\,h\in W_{r}^{1,p}([0,T],\mathbb{R}^{d})\}$ relative to $\|\cdot\|_{\alpha,r}$.

Now we let $h\in W_{r}^{1,p}([0,T],\mathbb{R}^{d})$ be fixed and recall that for any $n\in\mathbb{N}$ and each $x\in C([0,T],\mathbb{R}^{d})$ there is a unique mild solution $y_{h,n}\in C([0,T],\mathbb{R}^{d})$ to the ordinary integral equation with running value condition
\[
y_{h,n,x}(t) = x(t) - \int_{r}^{r\vee t}\dot{h}(s) - \dot{L}_{n}(y_{h,n,x})(s)\,ds\quad\text{for $t\in [0,T]$.}
\]
As the map $C([0,T],\mathbb{R}^{d})\rightarrow C([0,T],\mathbb{R}^{d})$, $x\mapsto y_{h,n,x}$ is Lipschitz continuous on bounded sets, we may let $\null_{h,n}W\in\mathscr{C}([0,T],\mathbb{R}^{d})$ be given by $\null_{h,n}W_{t}:=y_{h,n,W}(t)$ and introduce a martingale $\null_{h,n}Z\in\mathscr{C}([0,T],\mathbb{R})$ by requiring that $\null_{h,n}Z^{r} = 1$ and
\begin{equation*}
\null_{h,n}Z_{t} = \exp\bigg(\int_{r}^{t}\dot{h}(s)' - \dot{L}_{n}(\null_{h,n}W)(s)'\,dW_{s} - \frac{1}{2}\int_{r}^{t}|\dot{h}(s) - \dot{L}_{n}(\null_{h,n}W)(s)|^{2}\,ds\bigg)
\end{equation*}
for any $t\in [r,T]$ a.s. By Girsanov's theorem, $\null_{h,n}W$ is a $d$-dimensional $(\mathscr{F}_{t})_{t\in [0,T]}$-Brownian motion under the probability measure $P_{h,n}$ on $(\Omega,\mathscr{F})$ given by $P_{h,n}(A):=E[\mathbbm{1}_{A}\null_{h,n}Z_{T}]$ and $X$ is a strong solution to the stochastic Volterra integral equation
\begin{equation*}
X_{t} = X_{r} + \int_{r}^{t}b(t,s,X) + \sigma(t,s,X)\big(\dot{h}(s) - \dot{L}_{n}(\null_{h,n}W)\big)(s)\,ds + \int_{r}^{t}\sigma(t,s,X)\,d\null_{h,n}W_{s}
\end{equation*}
for all $t\in [0,T]$ a.s.~under $P_{h,n}$. Hence, let $\null_{n}Y$ be the unique strong solution to~\eqref{Sequence VIE} when $\underline{B} = b$, $B_{H} = \sigma$, $\overline{B} = -\sigma$ and $\Sigma=\sigma$ with $\null_{n}Y^{r} = \hat{x}^{r}$ a.s., then uniqueness in law implies that $P(\|\null_{n}Y - x_{h}\|\geq\varepsilon) = P_{h,n}(\|X-x_{h}\|_{\alpha,r}\geq \varepsilon)$ for any $\varepsilon > 0$. This shows that Theorem~\ref{Main Convergence Theorem} also yields~\eqref{Hoelder Limit 2} and the claimed representation follows.
\end{proof}

\let\OLDthebibliography\thebibliography
\renewcommand\thebibliography[1]{
  \OLDthebibliography{#1}
  \setlength{\parskip}{1pt}
  \setlength{\itemsep}{2pt}}

\bibliographystyle{plain}

\end{document}